\documentclass[leqno,11pt]{article}

\usepackage{amsmath,amsthm,amssymb,mathrsfs}
\usepackage{times}
\usepackage[pagebackref,hypertexnames=false]{hyperref}
\usepackage{fancybox,fancyhdr,graphics,epsfig}
\usepackage[usenames,dvipsnames]{color}
\usepackage{bbm,subcaption}
\usepackage{verbatim}
\usepackage[all]{xy}
\usepackage{authblk}

\usepackage{amsmath,amsthm,bbm,amssymb}
\usepackage[dvipsnames]{xcolor}
\DeclareMathOperator{\birth}{birth}
\DeclareMathOperator{\death}{death}

\DeclareMathOperator{\rank}{rank}

\DeclareMathOperator{\im}{Im}

\DeclareMathOperator{\logg}{\log\log}

\DeclareMathOperator{\PH}{PH}

\def\R{\mathbb{R}}
\def\F{\mathbb{F}}
\def\Z{\mathbb{Z}}

\def\T{\mathbb{T}}

\def\cO{\mathcal{O}}
\def\cP{\mathcal{P}}

\def\cV{\mathcal{V}}
\def\cS{\mathcal{S}}

\def\cX{\mathcal{X}}

\newcommand{\set}[1]{\left\{#1\right\}}
\newcommand{\sbrk}[1]{\left[#1\right]}

\newcommand{\param}[1]{\left(#1\right)}

\newcommand{\prob}[1]{\mathbb{P}\left(#1\right)}

\newcommand{\eps}{\varepsilon}

\newtheorem{lem}{Lemma}[section]
\newtheorem{thm}[lem]{Theorem}
\newtheorem{prop}[lem]{Proposition}
\newtheorem{cor}[lem]{Corollary}

\newtheorem{con}[lem]{Conjecture}
\newtheorem{rem}[lem]{Remark}

\newtheorem{observation}[lem]{Observation}

\theoremstyle{definition}

\newcommand{\cech}{\v{C}ech }

\newcommand{\erdren}{Erd\H{o}s-R\'enyi }

\newcommand{\ninf}{n\to\infty}
\newcommand{\pois}[1]{\mathrm{Poisson}\param{{#1}}}

\newcommand{\limninf}{\lim_{\ninf}}

\newcommand{\bs}{\backslash}

\numberwithin{equation}{section}

\def\bsplit#1\esplit{\begin{split} #1 \end{split} }
\def\splitb#1\splite{\begin{split} #1 \end{split} }
\def\beq#1\eeq{\begin{equation} #1 \end{equation}}
\def\eqb#1\eqe{\begin{equation} #1 \end{equation}}

\newcommand{\Hg}{\mathrm{H}}

\def\crito{\lambda_{c}}
\def\critv{\bar\lambda_{c}}
\def\io{ i_{k}}
\def\iv{\bar i_{k}}

\author[1]{Omer Bobrowski \thanks{omer@ee.technion.ac.il}}
\author[2]{Primoz Skraba \thanks{p.skraba@qmul.ac.uk}}
\affil[1]{Viterbi Faculty of Electrical Engineering\\Technion - Israel Institute of Technology}
\affil[2]{School of Mathematical Sciences \\Queen Mary University of London}

\title{Homological Percolation: The Formation of Giant $k$-Cycles}

\date{\today}

\begin{document}
\maketitle

\begin{abstract}
In this paper we introduce and study a higher-dimensional analogue of the giant component in continuum percolation. Using the language of algebraic topology, we define the notion of giant $k$-dimensional cycles (with $0$-cycles being connected components). Considering a continuum percolation model in the flat $d$-dimensional torus, we show that all the giant $k$-cycles ($1\le k \le d-1$) appear in the regime known as the thermodynamic limit. We also prove that the thresholds for the emergence of the giant $k$-cycles are increasing in $k$ and are tightly related to the critical values in continuum percolation. Finally, we provide bounds for the exponential decay of the probabilities of giant cycles appearing.
\end{abstract}
\section{Introduction}

Percolation theory focuses on the formation of large-scale structures, and originally introduced as a model for  propagation of liquid in porous media. The first percolation model, introduced by Broadbent and Hammersley \cite{broadbent1957percolation}, is known today as the \emph{bond-percolation} model where bonds (connection between sites) can be either open or closed independently at random. Since then, percolation theory has become one of the dominant areas in mathematics and statistical physics, see \cite{duminil-copin_sixty_2017} for a survey on the field. In this paper we focus on a \emph{continuum percolation} model that was introduced first by Gilbert \cite{gilbert_random_1961} as a model for ad-hoc wireless networks.
In continuum percolation,  geometric objects (grains) are deployed at random in space, and we consider the structure formed by their union (see \cite{meester_continuum_1996}).

We introduce a new higher-dimensional generalization of percolation phenomena using the language of algebraic topology.
In order to do so, we will be considering percolation in a finite (yet large) medium, where structures such as ``giant" connected components, one-arm events, and crossing components, may appear. Our main observation is that these formations are mostly \emph{topological} in nature, i.e.~they are concerned with qualitative aspects of connectivity rather than quantitative measures of geometry. In algebraic topology, connected components are considered ``$0$-dimensional cycles" (or rather equivalence classes of cycles), forming the first class in a sequence known as the \emph{homology groups}
$\set{\Hg_k}_{k\ge 0}$ (see Section \ref{sec:homology}). For example, elements in $\Hg_1$ ($1$-cycles) can be thought of as loops surrounding holes, elements in $\Hg_2$ ($2$-cycles) can be thought of as  surfaces enclosing cavities, and there is a general notion for $k$-dimensional cycles.
Our goal is to introduce a  notion of ``giant $k$-dimensional cycles," and explore the probability of these structures to appear.

We focus on the continuum percolation model where the grains are balls of a fixed (nonrandom) radius.
Suppose that we have a homogenous Poisson process $\cP_n$ with rate $n$ generated over a space $\cS$, and consider $\cO_r$ to be the union of balls of radius $r$ around $\cP_n$. We define as giant $k$-cycles in $\cO_r$ any $k$-cycle in $\cO_r$ that is also a $k$-cycle of $\cS$, i.e.~elements in the image of the map $\Hg_k(\cO_r)\to \Hg_k(\cS)$ (see Section \ref{sec:main} for formal definitions).
Note, that taking $\cS$ to be a box (or any compact and convex space) will be pointless, since a box has $\Hg_k=0$ (i.e.~no $k$-cycles) for all $k>0$. Instead, we focus on the $d$-dimensional flat torus $\T^d$, i.e.~the unit box $[0,1]^d$ with a periodic boundary. In this case, it is known that $\rank(\Hg_k(\T^d)) = \binom{d}{k}$, and therefore we expect to observe giant $k$-cycles emerging in $\cO_r$ for all $1\le k\le d$.

Our main result is the following. We define $E_k$ to be the event  that there \emph{exists} a giant $k$-cycle in $\cO_r$ (i.e.~$\im(\Hg_k(\cO_r)\to \Hg_k(\T^d)) \ne 0$), and $A_k$  to be the event  that $\cO_r$ contains \emph{all} giant $k$-cycles (i.e.~$\im(\Hg_k(\cO_r)\to \Hg_k(\T^d)) = \Hg_k(\T^d)$).
Similarly to the study of random geometric graphs \cite{penrose_random_2003}, we study the limit when $n\to\infty$ and $r = r(n)$ satisfies $nr^d = \lambda$ (i.e.~$r = (\lambda/n)^{1/d}$), known as the \emph{thermodynamic limit}. Our results show that there exist two sequences of threshold values $\lambda_{0,1}\le \cdots\le \lambda_{0,d-1}$, and $\lambda_{1,1}\le \cdots\le \lambda_{1,d-1}$, such that:
\begin{itemize}
  \item[(a)]  If $\lambda < \lambda_{0,k}$ then $\prob{E_k}\to 0$ exponentially fast.
  \item[(b)] If $\lambda > \lambda_{1,k}$ then $\prob{A_k}\to 1$ exponentially fast.
\end{itemize}
Clearly, $\lambda_{0,k}\le \lambda_{1,k}$ and we conjecture that these values are equal. We will prove equality for $k=1$, while for $k>1$ this remains an open problem.

\paragraph{Related work.} A few higher-dimensional percolation notions have been studied in the past. In \cite{aizenman_sharp_1983, grimmett_plaquettes_2010} the model of random plaquettes was studied, as a generalization for bond percolation models. The main idea here is to consider $\Z^d$, and instead of setting edges to be open/close, we do the same for the $k$-dimensional cubical faces ($2\le k \le d$).
The study in  \cite{aizenman_sharp_1983} focuses on $2$-dimensional plaquettes, and
asks  whether an arbitrarily large loop of edges in $\Z^d$ is covered by a $2$-dimensional surface of random plaquettes. Note, that loops that are not covered by any surface are exactly what we refer to as (nontrivial) $1$-cycles.
The results in \cite{grimmett_plaquettes_2010} consider $(d-1)$-dimensional plaquettes, and  address the formation of unoccupied spheres around the origin, which are related to entanglement. These spheres can also be thought of as $(d-1)$-cycles in homology. Using a similar cubical model, \cite{hiraoka_percolation_2018} studied percolation in a graph generated by neighboring $(d-1)$-cycles.

Another model is a generalization of the \erdren $G(n,p)$ random graph. Instead of a graph, we consider a simplicial complex (an object consisting of vertices, edges, triangles, tetrahedra and higher dimensional simplexes). The random $k$-complex $X_k(n,p)$ is generated by taking $n$ vertices, including all possible simplexes of dimensions $0,\ldots, k-1$, and setting the state of the $k$-simplexes independently as open with probability $p$, and closed otherwise.
In the $G(n,p)$ graph a giant component consisting of $\Theta(n)$ vertices is known to emerge when  $p=1/n$ \cite{erdos_evolution_1960}. The \emph{shadow} of a graph is the set of all edges not in the graph, whose addition to the graph generates a new cycle. When the giant component emerges, we observe that the shadow is \emph{giant}, i.e.~contains  a  fraction of the edges.
Similarly, in \cite{linial_phase_2016} it was shown that when $p=c/n$ for a known $c>0$, a giant $k$-shadow emerges in $X_k(n,p)$ where here the shadow refers to all the $k$-faces not in the complex, whose addition generates a new $k$-cycle.

The  phenomena above and the formation of giant $k$-cycles introduced in this paper are not obviously related. In particular, they are not directly comparable, as the plaquette model is infinite, the random $k$-complex has no underlying topology, and the torus we study here is both finite and has nontrivial homology.  Nontheless,  all these results describe different aspects of percolative behavior in higher dimensions. It is an interesting question whether any of these notions coincide for any particular model.

Finally, we note that the topological study of $\cO_r$ here is equivalent (via the Nerve Lemma \cite{borsuk_imbedding_1948}) to the study of the \emph{random \cech complex} \cite{kahle_random_2011}.
Recall that our results describe  the appearance of the giant $k$-cycles of the torus in the random process $\cO_r$.
A different transition related to $\cO_r$ (via the \cech complex), which can be referred to as \emph{homological connectivity}   \cite{bobrowski2019homological}, describes the stage where the $k$-th homology of $\cO_r$ not only contains all the $k$-cycles of the torus  ($\im(\Hg_k(\cO_r)\to \Hg_k(\T^d)) = \Hg_k(\T^d)$) but is completely \emph{identical} (isomorphic) to it (i.e.~$\Hg_k(\cO_r) \cong \Hg_k(\T^d)$). For the flat torus, it was shown in \cite{bobrowski2019homological} that a sharp phase transition for the $k$-th homological connectivity occurs when $nr^d = \frac{1}{\omega_d}(\log n + (k-1)\logg n)$, where $\omega_d$ is the volume of a unit ball in $\R^d$. Note that this regime is much denser than the thermodynamic limit we consider here ($nr^d = \text{const}$). For the top-dimensional homology ($\Hg_d$), the homological percolation and homological connectivity phase transitions are the same, and occur when $nr^d = \frac{1}{\omega_d}(\log n + (d-1)\logg n)$, which is also the coverage threshold. Hence, we do not consider the case of $k=d$ in this paper. With respect to the thermodynamic limit, we also note that several limit theorems have been proved in the past, for examples -- the Betti numbers \cite{yogeshwaran_random_2016}, the Euler characteristic \cite{thomas2019functional}, and the topological type distribution
\cite{auffinger_topologies_2018}.

\paragraph{Applied topology.}
While the study in this paper is mainly motivated and inspired by percolation theory, and the main goal is to seek higher dimensional analogue to percolation phenomena, we also want to highlight another interesting application of the results.

 The field of applied topology (or topological data analysis) promotes the use of mathematical topology in data and network analysis \cite{carlsson_topology_2009,edelsbrunner_computational_2010,ghrist2014elementary}
One of the most powerful tools  developed in this field is \emph{persistent homology} \cite{edelsbrunner_persistent_2008,zomorodian_topology_2005}.  Briefly, it is an algebraic tool that can be used to detect $k$-cycles that appear at different scales in observed data.
For example, given a point-cloud $\cX$, we can consider the filtration generated by the union of balls of varying radii $\set{B_r(\cX)}_{r=0}^\infty$. As we increase the radius, $k$-cycles can be formed (born), and later filled in (die). The $k$-th persistent homology, denoted $\PH_k(\cX)$ is the collection of all such $k$-cycles, where each cycle $\gamma$ is assigned with an interval $[\birth(\gamma), \death(\gamma))$ representing the range of scales (radii) in which the feature was observed, see Figure~\ref{fig:example1}.

\begin{figure}
  \centering
  \begin{subfigure}[t]{0.4\textwidth}
  \centering
 \includegraphics[width=0.715\textwidth]{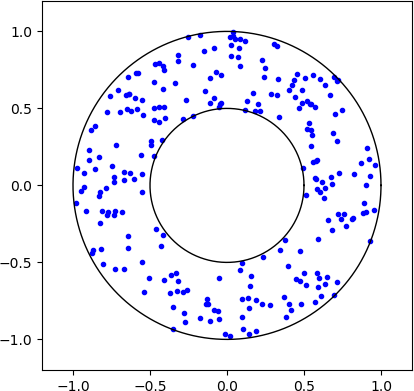}
  \caption{}
  \end{subfigure}
\hspace{25pt}
  \begin{subfigure}[t]{0.4\textwidth}
  \centering
  \includegraphics[width=0.9\textwidth]{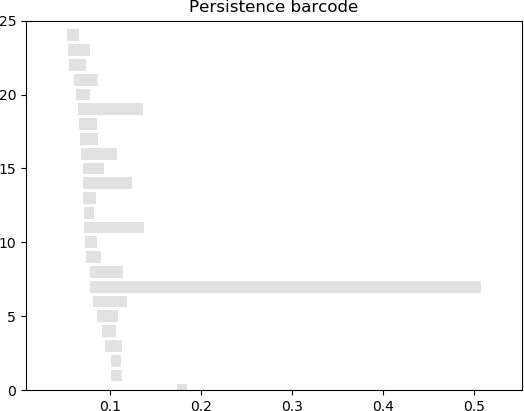}
  \caption{}
  \end{subfigure}
\caption{\label{fig:example1} Persistent homology for a random point cloud. (a) A random sample generated in an annulus whose inner radius is $0.5$ and outer radius is $1$. (b) We consider the persistent homology $\PH_1$ (i.e.~holes) generated by drawing balls of radius $r$ around the points, and increasing $r$. Each bar corresponds to a $1$-cycle, and its endpoints are the birth and death times (radii) of that cycle. Note that there is a single giant $1$-cycle here (representing the hole of the annulus), and its death time is roughly 0.5 (same as the inner radius).}
\end{figure}

One of the key problems in the field is to decide, among all the features in $\PH_k$, which ones represent statistically significant phenomena that one should look into (or the ``signal" underlying the data), and which are merely  artifacts of our finite sampling or other sources of randomness that should be ignored (``noise").  Over the years, several ideas have been proposed  (see the survey \cite{wasserman_topological_2018}), but none has grown into a robust statistical framework, so the problem is still very much open. The probabilistic analysis of persistent homology is highly challenging due to the potentially global and dependent nature of the algebraic-topological transformation. Nevertheless, over the past decade, some significant progress has been achieved \cite{adler_modeling_2017,hiraoka_limit_2018,owada2020convergence}. Considering \emph{individual}  cycles in $\PH_k$, the following theoretical result is the only one available to date.

For each $k$-cycle $\gamma\in \PH_k$ we can associate a measure of topological persistence by taking  $\pi(\gamma) := \death(\gamma)/\birth(\gamma)$. Suppose that the data $\cX$ are sampled over a space with a trivial homology (e.g.~a box, ball, etc.). In this case, all the cycles in $\PH_k$ should be considered as noise, since the signal is trivial. We define the extremal noise persistence as $\Pi_k(n) := \max_{\gamma\in \PH_k(\cX_n)} \pi(\gamma)$. The following result was proved in \cite{bobrowski_maximally_2017}.

\begin{thm}[\cite{bobrowski_maximally_2017}]\label{thm:max_cycles}
If $\cX = \cX_n$ is a homogeneous Poisson process with rate $n$,  then there exist $A,B > 0$, such that with high probability we have
\[
	A\param{\frac{\log n}{\logg n}}^{1/k} \le \Pi_k(n) \le B\param{\frac{\log n}{\logg n}}^{1/k}.
\]
\end{thm}

In other words, this result provides us with the asymptotic rate of the most persistent \emph{noisy} cycle. While \cite{bobrowski_maximally_2017} proved this result in a box, we note that this result will hold for any smooth compact manifold, as long $\Pi_k(n)$ is taken over the noisy cycles (i.e. ignoring precisely the giant cycles we study in this paper). With this result in hand, an obvious question is then -- how does the scaling in Theorem \ref{thm:max_cycles} compares to the persistence of the \emph{signal} (giant) cycles? Notice that the death of a signal cycle (in the limit) is  non-random, and depends on the geometry of the underlying space only. For example, sampling from an annulus, then the death time of the giant $1$-cycle is the inner radius  (asymptotically), see Figure \ref{fig:example1}. Therefore, in order to estimate the persistence ratio of the signal cycles, we need to evaluate their \emph{birth time}. The results in this paper provide the correct scaling for these birth times, and by that can be used to highlight the asymptotic differences between signal and noise in geometric models (see Section \ref{sec:discuss}).

\section{Preliminaries}\label{sec:prelim}

\subsection{Homology}\label{sec:homology}
Homology is an algebraic invariant which characterizes spaces and functions using groups and homomorphisms.
In a nutshell, if $X$ is a topological space, we have a sequence of groups $\set{\Hg_k(X)}_{k\ge 0}$, where loosely speaking, the generators of $\Hg_0(X)$ correspond to the connected components of $X$, the generators of $\Hg_1(X)$ correspond to
closed loops surrounding holes in $X$, $\Hg_2(X)$ corresponds to surfaces enclosing voids in $X$, and in general the elements of $\Hg_k(X)$ are considered nontrivial $k$-dimensional cycles. We refer the reader to ~\cite{hatcher_algebraic_2002} for more precise definitions, as for the most part we will not require it. We assume homology is computed using field coefficients, denoted by $\mathbb{F}$, and then the homology groups are simply vector spaces and the dimension of these vector spaces are known as the Betti numbers, denoted $\beta_k(X)$.

In this paper, we limit ourselves to the case of the $d$-dimensional torus.  The homology groups in this case are $\Hg_k \cong \F^{\binom{d}{k}}$, where $\F$ is the field of coefficients we use.
See Figure \ref{fig:torus}(a) for the case $d=2$. More concretely, we will study the \emph{flat torus} $\T^d = \R^d / \Z^d$, which is a topological torus with a locally flat metric. A useful way to think of $\T^d$ is using the unit box $Q^d = [0,1]^d$ with a periodic boundary, i.e.~$\T^d = Q^d / \{0\sim 1\}$. In this case, we can view the cycles in $\Hg_k(\T^d)$  as follows. Let $\gamma_{k,1} := \param{[0,1]^k\times \set{0}^{d-k}} / \set{0\sim 1}$, i.e.~we take a $k$-dimensional face of the $Q^d$, with the periodic boundary of the torus. Then each $\gamma_{k,1}$  introduces a $k$-dimensional cycle in $\T^d$. For each $k$, we can similarly generate a basis for $\Hg_k(\T^d)$ $\set{\gamma_{k,1},\gamma_{k,2},\ldots,\gamma_{k,\binom{d}{k}}}$ by taking $k$-faces of $Q^d$ in all possible  $\binom{d}{k}$ directions (i.e.~that are not parallel). We call these cycles the ``essential cycles" of the torus $\T^d$. See Figure \ref{fig:torus}(b). The careful reader should note that by cycle, we are referring to a cycle representative of a non-trivial homology class.

In addition to describing the properties of a single space $X$, homology groups can also be used to study functions between spaces. For a given function $f:X\to Y$ we have a collection of  induced maps $f_*:\Hg_k(X)\to \Hg_k(Y)$, that describe what happens to every $k$-cycle in $X$ after applying $f$. As $\Hg_k(X),\Hg_k(Y)$ are vector spaces, $f_*$ are linear transformations.

\begin{figure}[ht]
\centering
\begin{subfigure}{0.45\textwidth}
\centering
\includegraphics[scale=0.2]{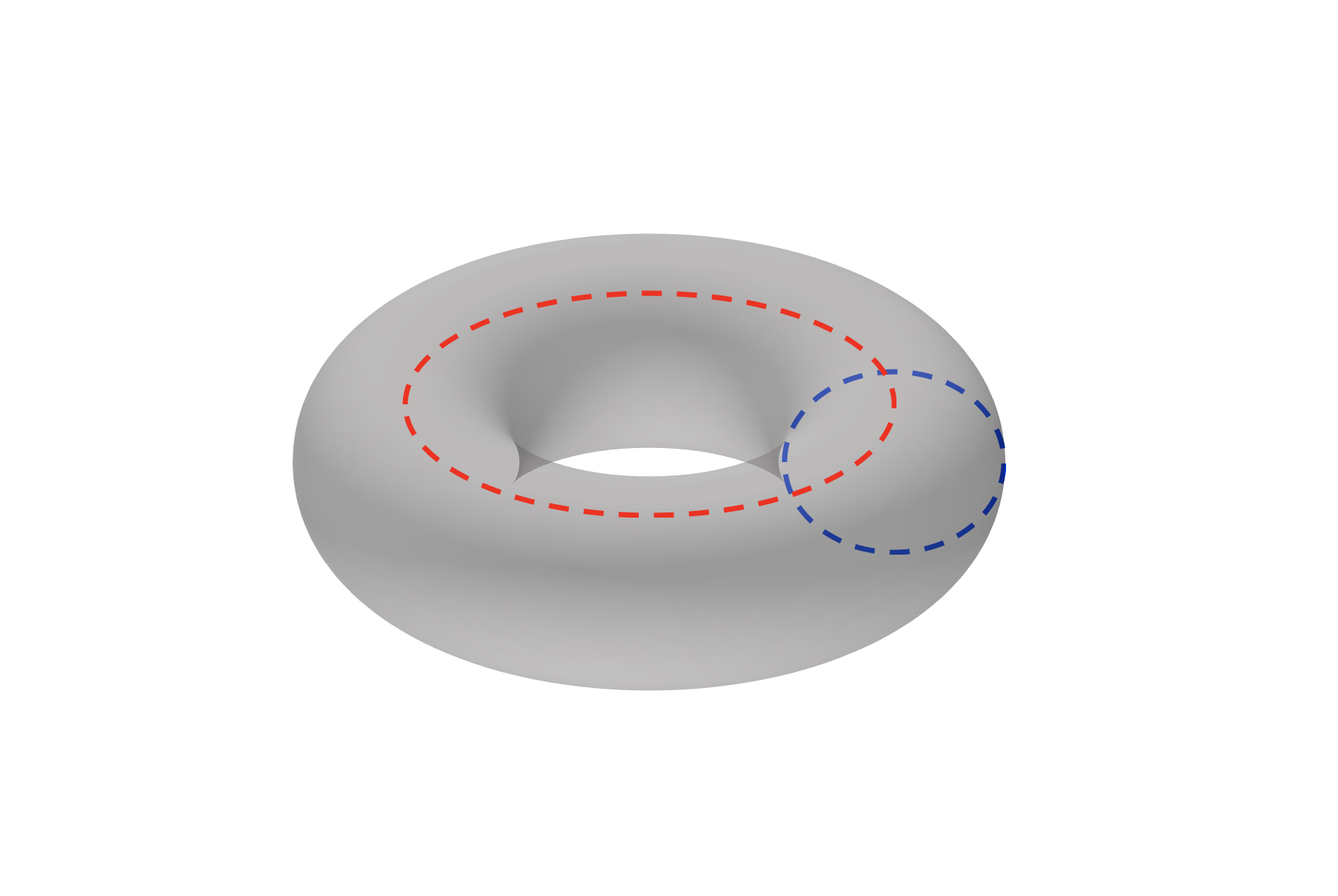}
\caption{}
\end{subfigure}
\begin{subfigure}{0.45\textwidth}
\centering
\includegraphics[scale=0.2]{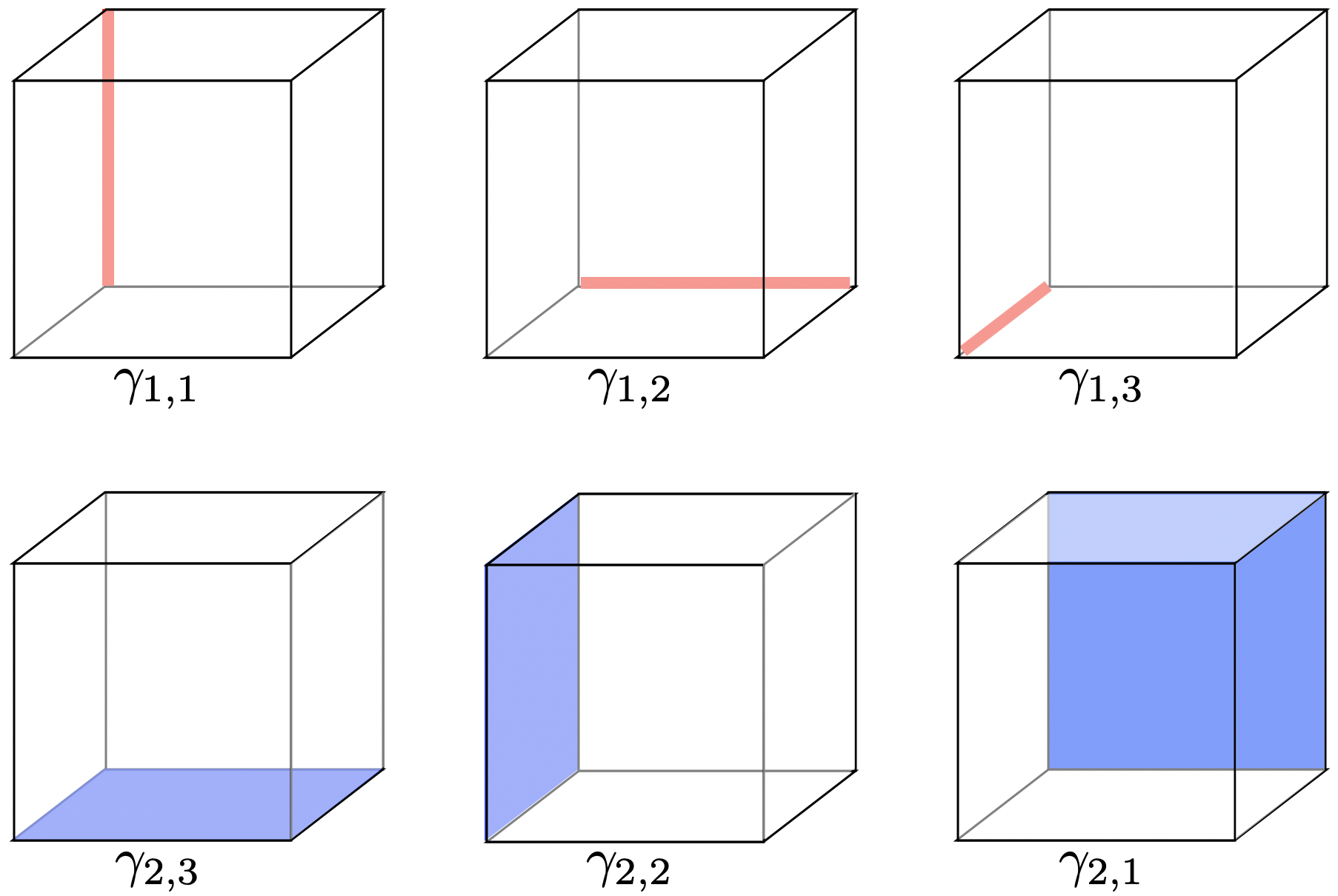}
\caption{}
\end{subfigure}
\caption{\label{fig:torus} The homology of the torus. (a) The $2d$ torus as a manifold. There is a single connected component -- $\Hg_0\cong \F$, two independent $1$-cycles (dashed line) -- $\Hg_1\cong \F^2$, and a single ``air pocket" -- $\Hg_2 \cong \F$. (b) The $3d$ flat torus $\T^3 = [0,1]^3/\set{0\sim 1}$, where $\Hg_1\cong \F^3$ and $\Hg_2\cong \F^3$. On the first row we mark the essential $1$-cycles, and on the second row the essential $2$-cycles. The columns are ordered so that the $2$-cycle at the bottom is the dual (via Lemma \ref{lem:duality}) of the $1$-cycle above.}
\end{figure}

\subsection{Continuum percolation}\label{sec:cont_perc}

Percolation theory focuses primarily on the formation of infinite components in random media.
In \emph{continuum percolation} (see \cite{meester_continuum_1996}), the medium is generated by geometric objects (grains) placed at random in space. In its simplest form, we have a homogeneous Poisson process in $\R^d$ with rate $\lambda$, denoted $\cP_\lambda$, and the grains are  fixed-size balls. We define the \emph{occupancy} and \emph{vacancy} processes as
\[
	\cO := \bigcup_{p\in \cP_\lambda} B_1(p)\quad\text{and}\quad \cV := \R^d \bs \cO,
\]
where $B_r(p)$ is the ball of radius $r$ around $p$.

The fundamental results in percolation theory are concerned with probability to form an infinite component. To this end, we define the event when the origin is part of an infinite component in $\cO$ as $I_0$, where `infinite' could refer to either the diameter, volume or the number of points (cf.~\cite{meester_continuum_1996}). Similarly, we define $\bar I_0$ for  the vacancy $\cV$. Next, we define the percolation probabilities
\[
	\theta(\lambda) := \prob{I_0},\quad \text{and}\quad \bar\theta(\lambda) := \prob{\bar I_0},
\]
and the percolation thresholds
\[
	\lambda_c := \inf\set{\lambda : \theta(\lambda) >0 },\quad\text{and}\quad \bar\lambda_c:= \sup\set{\lambda: \bar\theta(\lambda) > 0}.
\]
A fundamental result in continuum percolation then states that for all $d\ge 2$ we have
\[
	0 < \lambda_c \le \bar\lambda_c < \infty,
\]
with equality for $d=2$ \cite{roy1990russo}, and a strict inequality for $d>2$ \cite{sarkar_co-existence_1997}.

The critical values $\lambda_c,\bar\lambda_c$ were shown to control various phenomena related to the connected components of the occupancy/vacancy processes. Of a particular interest to us will be those related to crossing paths in a finite box. Define $W_n := [-\frac{n}{2}, \frac{n}{2}]^d$, and let $\cO(n),\cV(n)$ be the occupancy and vacancy processes generated by the points in $\cP_\lambda \cap W_n$. For every $n$, these processes are finite, and we can ask whether either of them contains a path that crosses the box from one side to the other. This question is interesting mainly in the limit as $n\to\infty$.

In the theory of \emph{random geometric graphs} \cite{penrose_random_2003}, an alternative and nearly equivalent model is studied, which will be useful for us in this paper. Instead of taking the growing box $W_n$ and a fixed radius $r=1$, we take the fixed unit box $Q^d = [0,1]^d$, consider the homogeneous Poisson process of rate $n$ in this box $\cP_n$, and study the processes
\eqb\label{eqn:occ_vac}
\cO_r := \bigcup_{p\in \cP_n} B_r(p),\quad\text{and}\quad \cV_r := Q\bs \cO_r.
\eqe
To make the models equivalent we set $n_\lambda = (n/\lambda)^{1/d}$, and note that the limiting behavior of the processes $\cO(n),\cV(n)$ (in $W_n$) is same as $\cO(n_\lambda),\cV(n_\lambda)$ (in $W_{n_\lambda}$, for any fixed $\lambda >0$). In addition, by a scaling argument -- taking $\cP_\lambda$ in $W_{n_\lambda}$ with balls of radius $1$ is equivalent to taking $\cP_n$ in $Q^d$ with balls of radius $r=(\lambda/n)^{1/d}$. In other words, the processes $\cO(n_\lambda)$ and $n_\lambda \cO_r$ have the same distribution (up to translation).
To conclude, we will consider the model in \eqref{eqn:occ_vac}, under the condition
\eqb\label{eqn:TD}
nr^d = \lambda,
\eqe
for a fixed $\lambda \in (0,\infty)$.
Note that this implies that $r = (\lambda/n)^{1/d} \to 0$ as $n\to \infty$.

To prove our main result, we will need the following statements that are adapted from the continuum percolation literature.
For any two sets $A,B\subset Q$ we denote by $A \stackrel{\cO_r}{\longleftrightarrow}B$ the event that there exists a path in the occupancy process that connects a point in $A$ to a point in $B$. Similarly, we define $A \stackrel{\cV_r}{\longleftrightarrow}B$ for the vacancy processes.

The first statements we need are about the exponential decay of the one-armed probabilities.

\begin{prop}\label{prop:exponential_decay}
Let $c$ be the center point of the cube $Q$, and $\partial B_{R}(c)$ be the boundary of the ball of radius $R<1/2$ centered at $c$.\\
If $\lambda < \lambda_c$, there exists $C_1>0$ (possibly depends on $\lambda$) such that
\[
	\prob{c \stackrel{\cO_r}{\longleftrightarrow} \partial B_{R}(c)} \le e^{-C_1 R n^{1/d}}.
\]
If $\lambda > \bar\lambda_c$, there exists $C_2>0$ (possibly depends on $\lambda$) such that
\[
	\prob{c \stackrel{\cV_r}{\longleftrightarrow} \partial B_{R}(c)} \le e^{-C_2 R n^{1/d}}.
\]
\end{prop}
\begin{proof}
This is merely a scaled version of Theorem 2 and 4 in \cite{duminil-copin_subcritical_2018}. In \cite{duminil-copin_subcritical_2018} it is proved that if $\lambda<\lambda_c$ then for any $\tilde R>0$ we have
\[
	\prob{0\stackrel{\cO}\longleftrightarrow \partial B_{\tilde R}(0)} \le e^{-c_\lambda \tilde R},
\]
for some $c_\lambda>0$. Scaling  by $r$ and shifting by $c$, we have
\[
	\prob{c\stackrel{\cO_r}\longleftrightarrow \partial B_{\tilde Rr}(c)} \le e^{-c_\lambda \tilde R}.
\]
Finally, since $r = (\lambda/n)^{1/d}$ we have
\[
	\prob{c\stackrel{\cO_r}\longleftrightarrow \partial B_{R}(c)} \le e^{-c_\lambda R/r} = e^{-C_1  R n^{1/d}}.
\]
Similarly, we can prove the statement for $\cV_r$.

\end{proof}
%
The next statement we need is about the crossing paths and uniqueness of the giant component.

\begin{prop} \label{prop:penrose_sup}
Suppose that $d\ge 2$ and $\lambda > \lambda_c$. Take any $D \le 1$. Denote by $E$ the events that:
\begin{enumerate}
\item There exists a unique component of $\cO_r$ that crosses the box $Q$ in \emph{all} directions.
\item The diameter of all other components in $\cO_r$ is at most $D$.
\end{enumerate}
Then there exists $C_3>0$ (possibly depends on $\lambda$), so that
\[
	\prob{E} \ge 1-e^{-C_3 D n^{1/d}}.
\]
\end{prop}

\begin{proof}
We use a scaled version of Proposition 2 in \cite{penrose_large_1996}. For the cube $W_n$ it is proved in \cite{penrose_large_1996} that when $\lambda > \lambda_c$,  for any $\log n \ll \phi_n \le n$ the probability that there exists a unique giant component in $\cO(n)$, crossing the box $W_n$ in all directions, and that all other components have diameter less than $\phi_n$ is bounded from below by $1-e^{-C_3\phi_n}$ for some $C_3>0$. Scaling from $\cO(n)$ to $\cO_r$, then implies that for $n^{-1/d}\log n \ll D \le 1$ we have $\prob{E} \ge 1-e^{-C_3 D n^{1/d}}$.
\end{proof}

\begin{rem}
While Theorem \ref{prop:penrose_sup} is stated for a cube, the  proof in \cite{penrose_large_1996} applies for a box of any fixed-size dimensions.
\end{rem}


\section{Main result}\label{sec:main}

Throughout this paper we let $d\ge 2$ be fixed and consider the \emph{flat torus} $\T^d = \R^d / \Z^d$ (see Section \ref{sec:homology}).
Let $\set{X_1,X_2,\ldots}$ be iid random variables uniformly distributed in $\T^d$, let $N\sim \pois{n}$ be another independent variable, and define $\cP_n := \set{X_1,\ldots, X_N}$.
Then $\cP_n$ is  a homogeneous  Poisson process on $\T^d$ with rate $n$, and define the occupancy and vacancy processes as above by
\[
\cO_r := \bigcup_{p\in \cP_n} B_r(p),\quad\text{and}\quad\cV_r := \T^d \bs \cO_r,
\]
where we use balls with respect to the toroidal metric (which is locally flat).

In order to define the giant $k$-cycles, we consider the inclusion maps
$i:\cO_r \to \T^d$ and $\bar i: \cV_r\to \T^d$
and consider their induced maps (homomorphisms) in homology,
\[
\io : \Hg_k(\cO_r)\to \Hg_k(\T^d),\quad\text{and}\quad  \iv : \Hg_k(\cV_r)\to \Hg_k(\T^d).
\]
Loosely speaking, the image of $\io$ (resp.~$\iv$) corresponds to the $k$-cycles of the tours that have a representative element in $\cO_r$ (resp.~$\cV_r$). We define the $k$-th \emph{homological percolation} events as
\[
	E_k := \{\im(\io) \ne 0\},\quad \text{and}\quad A_k := \{\im(\io) = \Hg_k(\T^d)\},
\]
and similarly for vacancy we define $\bar E_k, \bar A_k$.
The event $E_k$ asserts that at least one of the $k$-cycles of the torus is represented in $\cO_r$, while $A_k$ asserts that all  of them are.

 In Figure~\ref{fig:experiments_1}, we observe $1$-cycles that realize the event $A_1$ in the 2-dimensional torus and the 3-dimensional torus. In Figure~\ref{fig:experiments_2} we show the $2$-cycles that realize $A_2$ for the 3-dimensional torus. One important remark is that the examples illustrate that  the giant cycles need not be simple, e.g. top-to-bottom or left-to-right.

\begin{figure}
  \centering
  \begin{subfigure}{0.4\textwidth}
  \centering
  \includegraphics[width=0.85\textwidth]{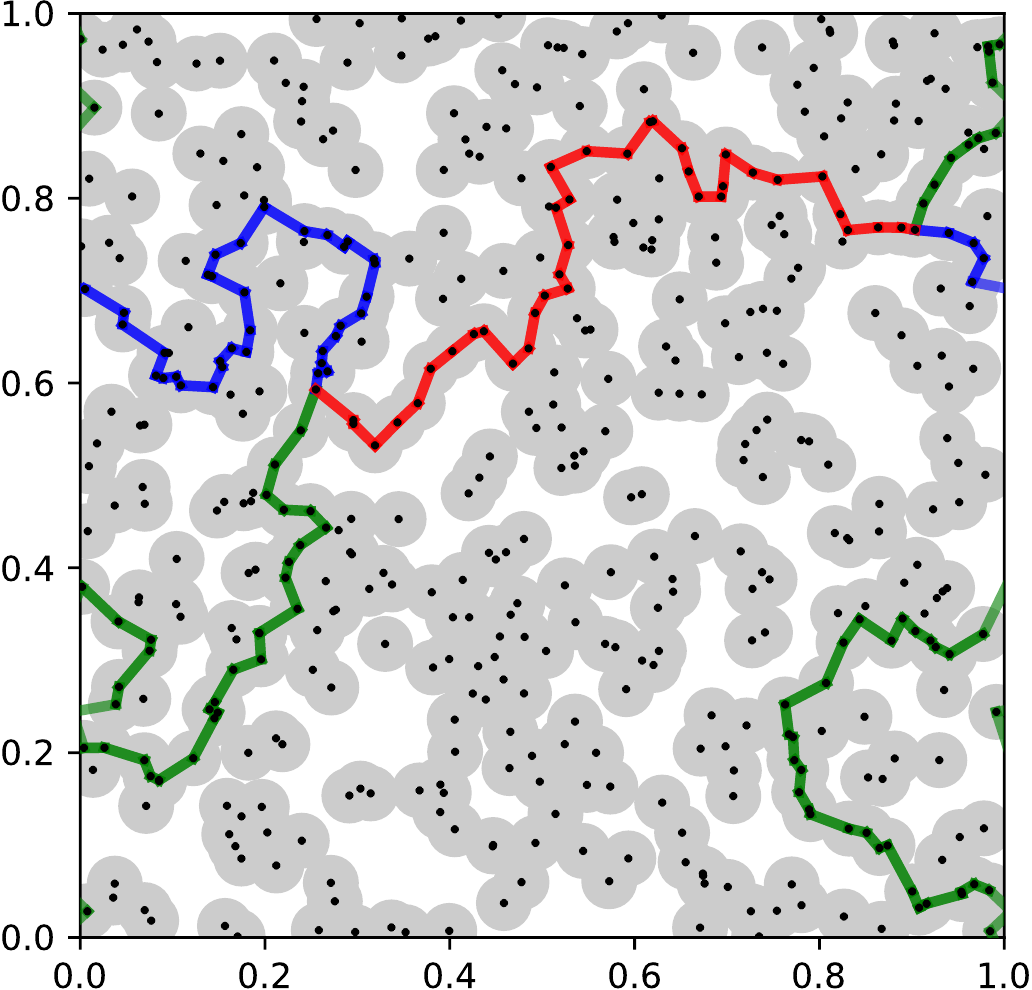}
  \caption{}
  \end{subfigure}
  \hspace{25pt}
  \begin{subfigure}{0.4\textwidth}
  \centering
  \includegraphics[width=\textwidth]{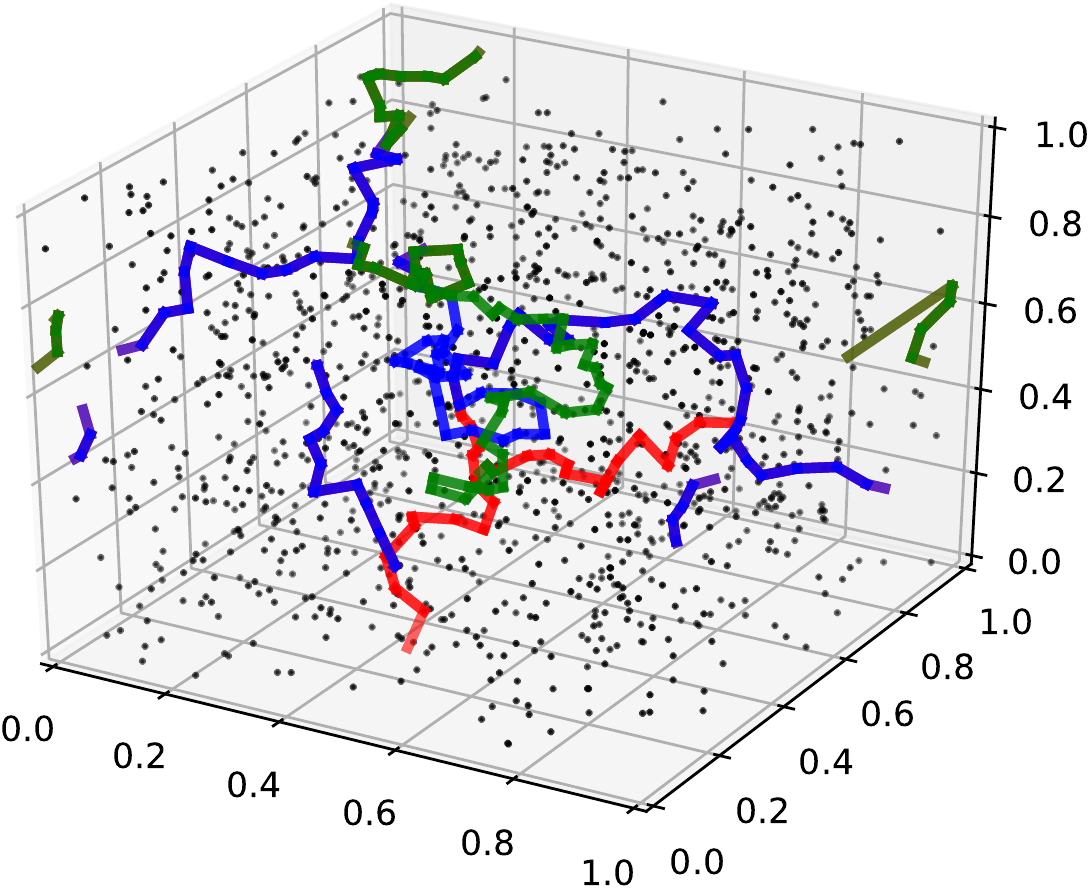}
  \caption{}
  \end{subfigure}
  \caption{\label{fig:experiments_1} The formation of giant $1$-cycles in the flat torus. (a) We plot realizations of the $1$-cycles generated by random balls in $\T^2$ (box with periodic boundary), where we have two giant cycles. The first cycle consists of the green+red paths, and the second cycle consists of the blue+red paths. (b) We plot the $1$-cycles generated in $\T^3$, where we have three of them. To simplify the picture we do not show the balls here, only the paths  (red,gree,blue) that correspond to the =$1$-cycles (note that the cycles may overlap).}
 \end{figure}

\begin{figure}
  \centering
  \begin{subfigure}{0.3\textwidth}
  \centering
  \includegraphics[width=0.9\textwidth]{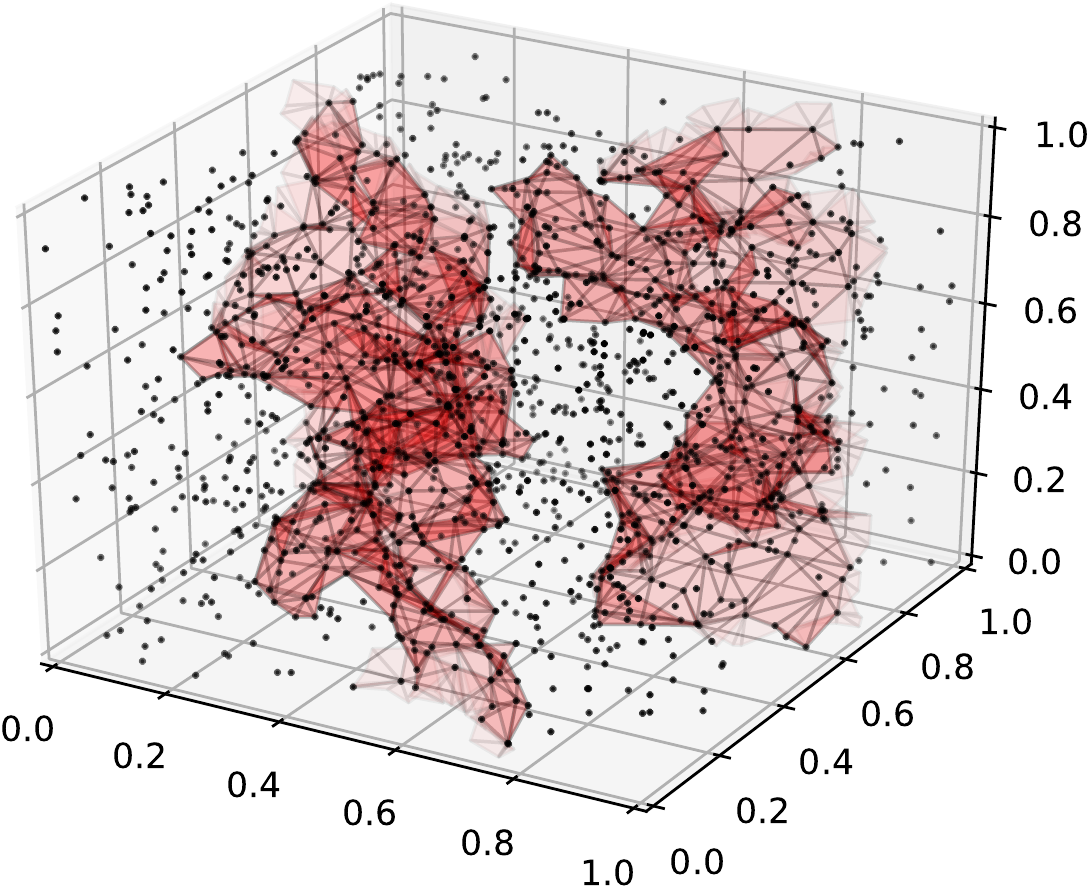}
  \caption{}
  \end{subfigure}
  \begin{subfigure}{0.3\textwidth}
  \centering
  \includegraphics[width=.9\textwidth]{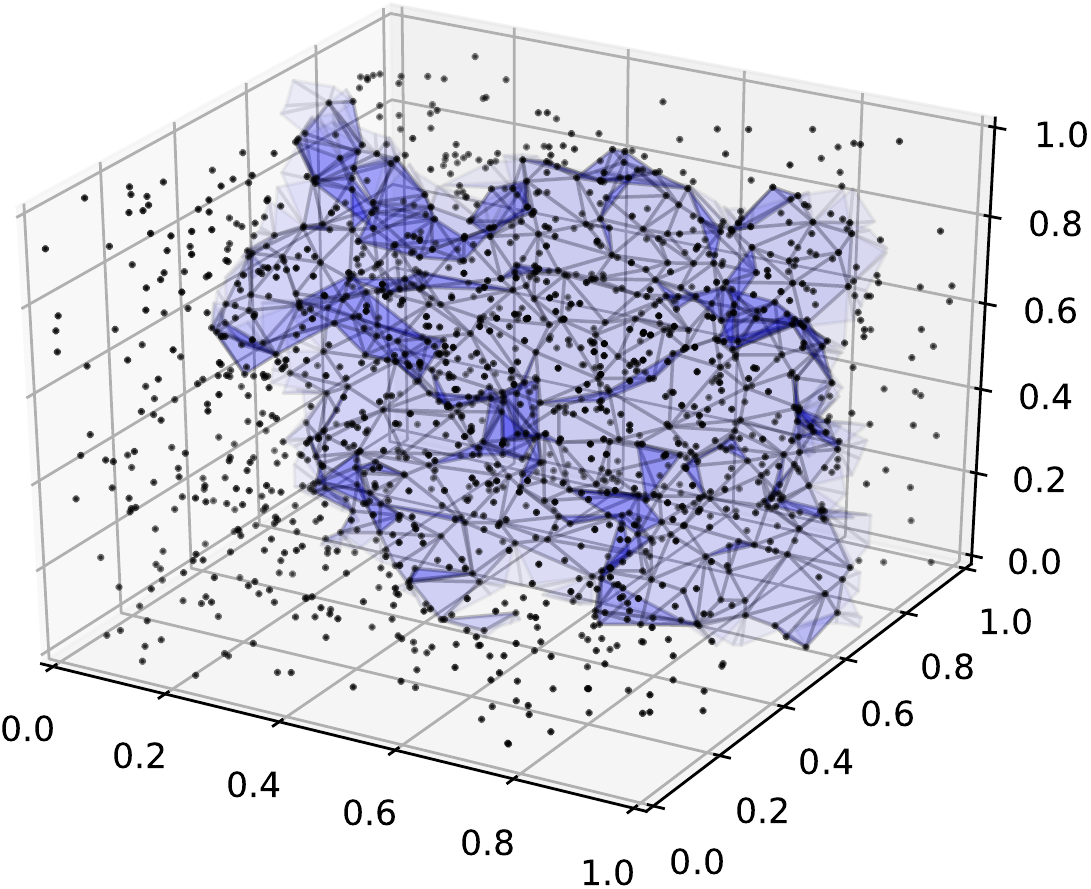}
  \caption{}
  \end{subfigure}
 \begin{subfigure}{0.3\textwidth}
  \centering
  \includegraphics[width=.9\textwidth]{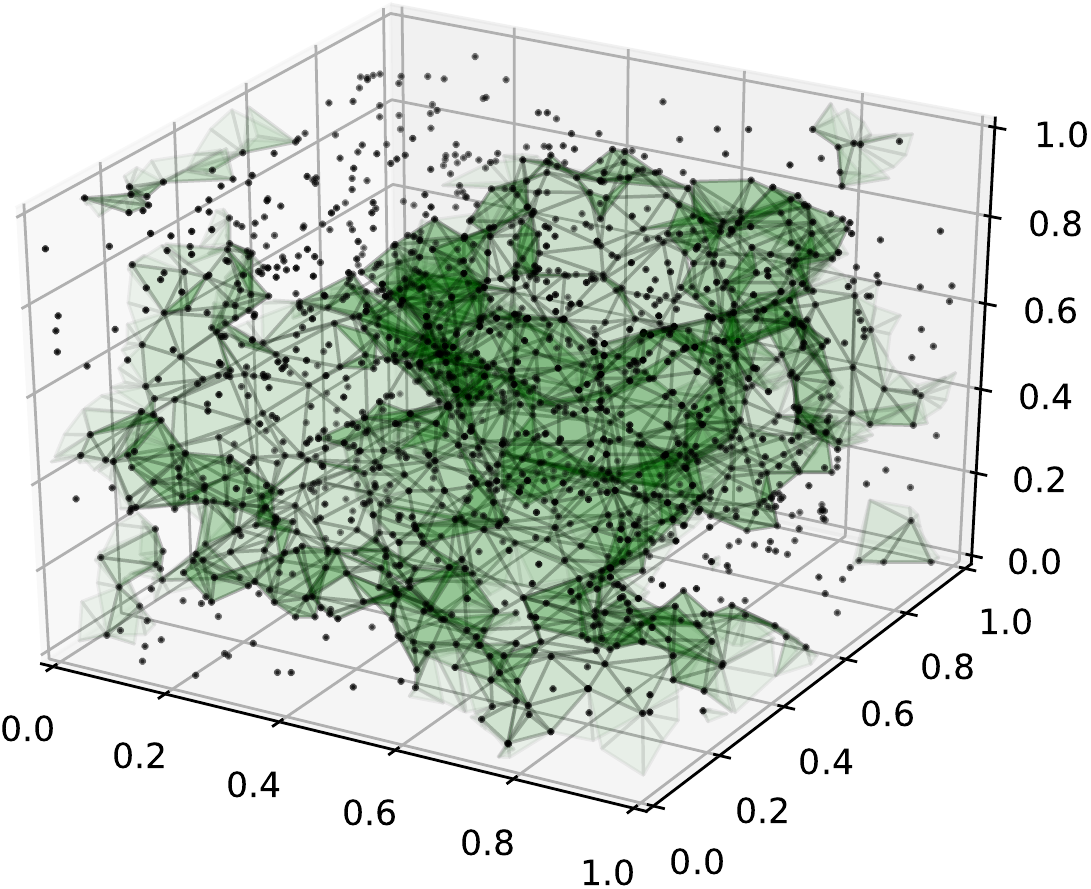}
  \caption{}
  \end{subfigure}
  \caption{\label{fig:experiments_2} The formation of giant $2$-cycles in the flat torus. Here we take $\T^3$ ($[0,1]^3$ with periodic boundary) and draw the giant $2$-cycles formed by the union of balls over a random sample. To simplify the picture we only show a triangulated version (the nerve) of the balls generating the $2$-cycles. Each of the 2d surfaces presented is a $2$-cycles, meaning that it encloses a cavity in the structure.}
 \end{figure}

 We can now state the  main result of the paper, which considers the limiting probability of the events $E_k,A_k$ as $n\to \infty$.
\begin{thm}\label{thm:main}
Let $d\ge 2$, and let $nr^d = \lambda$. Then there exist two sequences
\eqb\label{eqn:lambdas}
0<\lambda_{0,1} \le \lambda_{0,2}  \le \cdots \le \lambda_{0,d-1} < \infty,\quad \text{and}\quad 0<\lambda_{1,1} \le \lambda_{1,2}  \le \cdots \le \lambda_{1,d-1} < \infty,
\eqe
with $\lambda_{0,k} \le \lambda_{1,k}$, such that the following holds.\\
If $\lambda < \lambda_{0,k}$ then
\eqb\label{eqn:lim1}
\limsup_{n\to\infty} n^{-1/d}\log \prob{A_k} \le \limsup_{n\to\infty}  n^{-1/d}\log \prob{E_k}  < 0,
\eqe
and if $\lambda > \lambda_{1,k}$ then
\eqb\label{eqn:lim2}
\limsup_{n\to\infty}  n^{-1/d}\log{(1-\prob{E_k})} \le \limsup_{n\to\infty}  n^{-1/d}\log{(1-\prob{A_k})} < 0.
\eqe
Further, we have that $\lambda_{0,1} = \lambda_{1,1} = \crito$, and $\lambda_{1,d-1} \le \critv$, where $\lambda_c,\bar\lambda_c$ are the critical values for continuum percolation discussed in Section \ref{sec:prelim}.
\end{thm}

In other words, the theorem implies that there exist $C_{0,k},C_{1,k}>0$ (possibly depending on $\lambda$) such that  for $\lambda < \lambda_{0,k}$ and for large enough $n$ we have
\[
\prob{A_k}\le \prob{E_k} \le e^{-C_{0,k}n^{1/d}},
\]
and for $\lambda > \lambda_{1,k}$, for large enough $n$,
\[
\prob{E_k} \ge \prob{A_k} \ge 1-e^{-C_{1,k} n^{1/d}}.
\]
This, in particular, implies that
\[
	\limninf \prob{E_k} = 	\limninf \prob{A_k} = \begin{cases} 1 & \lambda > \lambda_{1,k},\\
	0 & \lambda < \lambda_{0,k}. \end{cases}
\]
The main conclusion from Theorem \ref{thm:main} is that the giant cycles of all dimensions $0< k<d$ appear within the \emph{thermodynamic limit} (i.e.~$nr^d = \text{const}$). This observation is not so obvious, since forming giant $k$-cycles require $\cO_r$ to cover large $k$-dimensional surfaces, while the process itself is still relatively sparse. For example, \emph{homological connectivity} -- the phase when $\Hg_k(\cO_r) \cong \Hg_k(\T^d)$, occurs at a much later stage, when $nr^d \sim \log n$ \cite{bobrowski2019homological}. Note that $k=d$ is excluded from the theorem. The $d$-cycle of the torus can only appear in $\cO_r$ upon coverage, which also occurs when $nr^d \sim \log n$ \cite{bobrowski2019homological}.

Another conclusion from the theorem is that the appearance of the giant $k$-cycles occurs in an orderly fashion, increasing in $k$. In addition, all the cycles are formed in the interval $[\crito,\critv]$. Further, once the giant component in $\cO_r$ appears (at $\crito$) it already includes (w.h.p.) all the giant $1$-cycles, and hence $\lambda_{0,1} = \lambda_{1,1} =  \crito$. This behavior will be made clearer in the proof.

Note that Theorem \ref{thm:main} provides a sharp phase transition only for the case of $k=1$, and the inequalities between the thresholds are not strict. However, since sharpness is a key property in most percolation models \cite{aizenman1987sharpness,menshikov1986coincidence,duminil-copin_sharp_2017}, we  believe that a stronger statement is true here as well. The proof of this statement will remain as future work.
\begin{con}\label{conj}
For all $0<k<d$ we have $\lambda_{0,k} = \lambda_{1,k} := \lambda_k$, and in addition
\[
\lambda_c = \lambda_1 < \lambda_2 < \cdots < \lambda_{d-1} = \bar \lambda_c.
\]
\end{con}

\section{Proofs}

In this section we prove Theorem \ref{thm:main}. We start by defining
\eqb\label{eqn:lambda_k}
\splitb
\lambda_{0,k} &= \sup\set{\lambda : \limsup_{\ninf} n^{-1/d}\log\prob{E_k} < 0}, \\
 \lambda_{1,k} &= \inf\set{\lambda : \limsup_{\ninf} n^{-1/d} \log(1-\prob{A_k})<0}.
\splite
\eqe
In case the first set is empty, we set $\lambda_{0,k }= -\infty$, and in case the second set is empty we set $\lambda_{1,k} = \infty$ (we will show later that neither set is empty).
Note that by definition, \eqref{eqn:lim1} and \eqref{eqn:lim2} hold. From the definitions we also have $\lambda_{0,k}\le \lambda_{1,k}$ for all $k$, since if $\prob{A_k}\to 1$ then surely $\prob{E_k} \not\to 0$.
 Thus, in order to prove Theorem \ref{thm:main} we have to show that all thresholds are in $(0,\infty)$ and are increasing in $k$ as in \eqref{eqn:lambdas}.
We will break the proof of Theorem \ref{thm:main} into three parts.  We start by proving that $\lambda_{0,1} = \lambda_{1,1} = \lambda_c > 0$. Next, we prove that $\lambda_{1,d+1}\le \bar \lambda_c < \infty$. Finally, we prove that for all $1\le k \le d-2$ we have $\lambda_{0,k} \le \lambda_{0,k+1}$, and $\lambda_{1,k} \le \lambda_{1,k+1}$.  That will conclude the proof.

\subsection{Giant $1$-cycles}

Our goal in this section is to prove the following lemma.

\begin{lem}\label{lem:proof_1}
The thresholds for the giant $1$-cycles satisfy
 $\lambda_{0,1} = \lambda_{1,1} = \lambda_c$.
\end{lem}

\begin{proof}
Suppose first that $\lambda < \crito$. Recall that we can consider the torus $\T^d$ as the quotient $Q^d/\set{0\sim 1}$, and take a discretization of $Q^d$ by the grid $\eps r\cdot\mathbb{Z}^d$, where $\eps$ is chosen small enough that any ball of radius
$r$
intersects at least one grid point (i.e. $\eps<1/\sqrt{d}$).

Suppose that $\im{i_1}\ne 0$, i.e.~there exists a non-trivial $1$-cycle in $\Hg_1(\cO_r)$ that is mapped to a non-trivial $1$-cycle in $\Hg_1(\T^d)$. Denote by $\gamma$ one of the (possibly many) combinations of balls in $\cO_r$ that realizes this cycle.   Since $\gamma$ contains at least one ball,  it must intersect with at least one of the grid points, denoted $x_0$. In addition, fixing $R < 1/2$, then the ball $B_R(x_0) \subset\T^d$ is contractible, and therefore its homology is trivial. Thus, any cycle supported on a component that is contained in $B_R(x_0)$ will be mapped to a trivial cycle in $\T^d$ (see Figure~\ref{fig:small_comps}).  We therefore conclude that $\gamma$ must intersect with the boundary $\partial B_R(x_0)$.
In other words, there must be a path in $\cO_r$ connecting $x_0$ to  $\partial B_R(x_0)$.  By Proposition~\ref{prop:exponential_decay}, and the translation invariance of the torus, this occurs with probability at most  $e^{-C_1 R n^{1/d}}$. Since there are $M = (\eps r)^{-d} = O(n)$ many grid points,  taking a union bound we conclude that
\[
	\prob{E_1} \le M e^{-C_1 R n^{1/d}} = O\param{n e^{-C_1 R n^{1/d}}}.
\]
Thus, we conclude that $\lambda \le \lambda_{0,1}$. Since we assumed $\lambda < \lambda_c$, we have $\lambda_{0,1}\ge \lambda_c$.

\begin{figure}
\centering
 \begin{subfigure}{0.45\textwidth}
  \centering
  \includegraphics[width=.6\textwidth]{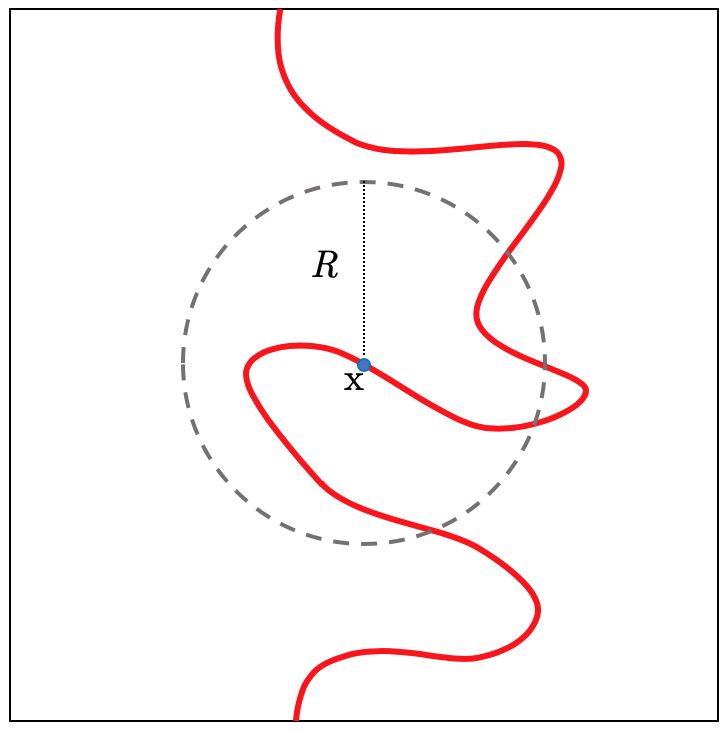}
  \caption{}
  \end{subfigure} \begin{subfigure}{0.45\textwidth}
  \centering
  \includegraphics[width=.6\textwidth]{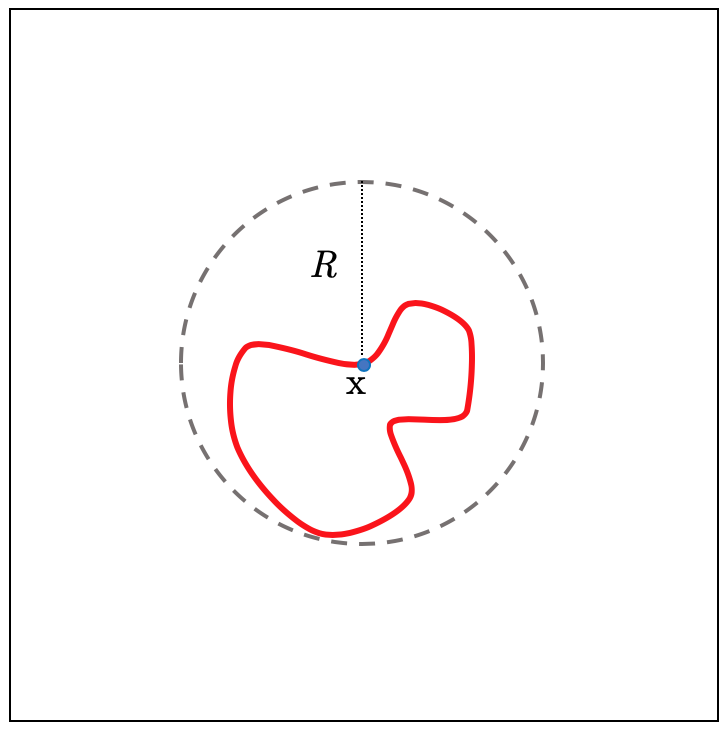}
  \caption{}
  \end{subfigure}\caption{\label{fig:small_comps}
  (a) The path presented here generates a giant $1$-cycle in $\cO_r$. The point $x$ is on this cycle, and we can see that there is a path from $x$ to $\partial B_R(x)$. (b) Here we have a $1$-cycle in $\cO_r$ that is fully contained in $B_R(x)$. Indeed, this not a giant $1$-cycle, since in $\T^d$ this loop does not surround a hole (i.e.~it is a boundary).}
\end{figure}

Next, suppose that $\lambda > \crito$. Note that we can also think of the torus $\T^d$ as $\T^d = ([0,4/3]\times[0,1]^{d-1}) / \Z^d$. With this in mind, we define the boxes
\[
	R_i = \param{\sbrk{i/3, (i+2)/3}\times [0,1]^{d-1}} / \Z^d \subset \T^d,\quad i=0,1,2,
\]
as well as their  intersections $R_{i,j} := R_i\cap R_j$.
For each of the boxes $R_i$, we can define a Poisson process $\cP_n^{(i)} = \cP_n\cap R_i$.
Next, we define the occupancy process $\cO_r^{(i)}$ as the union of $r$-balls around $\cP_n^{(i)}$ in $R_i$ with its Euclidean (rather than toroidal) metric.
Denote by $B_i$ the event that
\begin{enumerate}
\item The process $\cO_r^{(i)}$ contains a path crossing $R_i$ along its short (2/3) side.
\item There is a unique component in $\cO_r^{(i)}$ whose diameter is larger than $1/6$.
\end{enumerate}
According to Theorem \ref{prop:penrose_sup}, we have that $\prob{B_i} \ge 1-e^{-\frac{1}{6}C_3n^{1/d}}$. Using a union bound we then have that
\[
	\prob{B_1\cap B_2 \cap B_3} \ge 1-3e^{-\frac{1}{6}C_3 n^{1/d}}.
\]
Under the event $B = B_1\cap B_2\cap B_3$, we denote by $L_i$ the largest component in $\cO_r^{(i)}$, so that it contains a crossing path on the short side, denoted $\pi_i$. Note that each  $\pi_i$ also contains a path crossing $R_{i,j}$ ($j\ne i$) along the shorter side, denoted $\pi_{i,j}$. While $\pi_{i,j}$ is not necessarily contained in $\cO_r^{(j)}$, it is true that the diameter of $\pi_{i,j}\cap \cO_r^{(j)}$ is at least $1/3-r > 1/6$. Therefore, we conclude that $\pi_{i,j}\cap \cO_r^{(j)}\subset L_j$, implying that there is a path in $\cO_r^{(j)}$ connecting $\pi_{i,j}$ and $\pi_j$. To conclude, under the event $B$ we  have the following sequence of connected paths,
\[
	\pi_1 \stackrel{\cO_r^{(1)}}\longrightarrow \pi_{1,2} \stackrel{\cO_r^{(2)}}\longrightarrow \pi_2 \stackrel{\cO_r^{(2)}}\longrightarrow \pi_{2,3} \stackrel{\cO_r^{(3)}}\longrightarrow \pi_3 \stackrel{\cO_r^{(3)}}\longrightarrow \pi_{3,1}\stackrel{\cO_r^{(1)}}\longrightarrow \pi_1.
\]
In other  words, we showed that under $B$ we can find a path in $\cO_r$ that loops along one of the sides of the torus. Such a loop will generate an element in $\Hg_1(\cO_r)$ that is homologous to the essential $1$-cycle of the torus $\gamma_{1,1}$ (see Section \ref{sec:homology}).  Repeating the same arguments in all $d$-directions, and using a union bound will imply that
\[
	\prob{A_k} \ge 1-3d e^{-\frac{1}{6}C_3 n^{1/d}}.
\]
Thus, we must have $\lambda \ge \lambda_{1,1}$, and since $\lambda >\lambda_c$ we conclude that $\lambda_{1,1}\le \lambda_c$.

Finally, we showed that $\lambda_{1,1}\le \lambda_c \le \lambda_{0,1}$. On the other hand, \eqref{eqn:lambda_k} implies that $\lambda_{0,1}\le \lambda_{1,1}$. Thus, we conclude that $\lambda_{0,1}=\lambda_{1,1} = \lambda_c$, concluding the proof.

\begin{figure}
\centering
\includegraphics[width=0.6\textwidth]{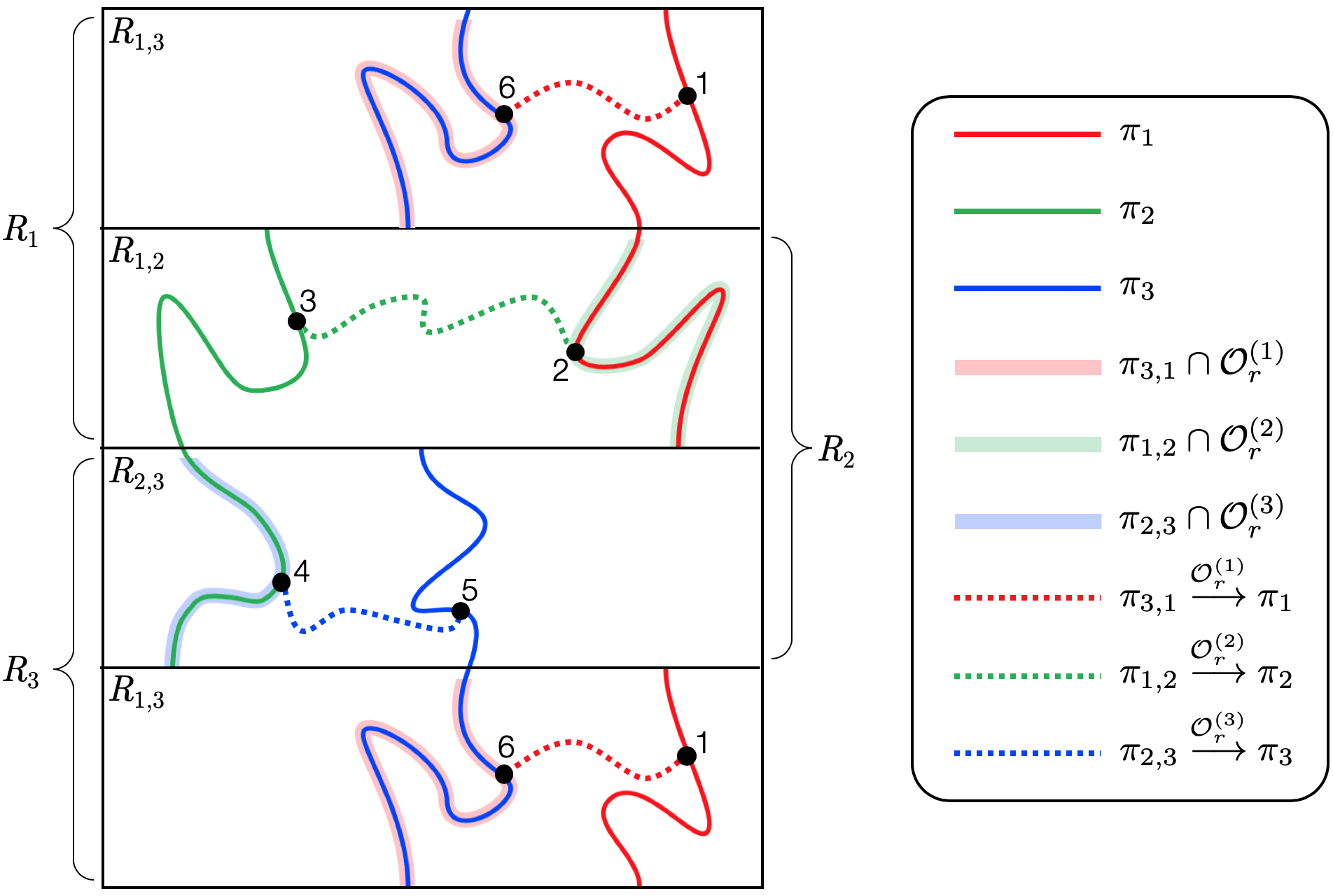}
\caption{\label{fig:h1_perc} Considering the torus $\T^d$ as the quotient $[0,4/3]\times[0,1]^{d-1}/\Z^d$. We then split the torus into the boxes $R_1,R_2,R_3$ and their intersection. Notice that the top and bottom rectangle are identical ($R_{1,3}$). Using the gluing arguments in the proof, and connecting the dots from 1 to 6, we get a loop that generates the top-bottom $1$-cycle in the picture.
Note that as mentioned in the proof, the paths $\pi_{i,j}\cap \cO_r^{(j)}$ are not necessarily crossing for $R_{i,j}$, as can be seen in the figure.}
\end{figure}
\end{proof}
\begin{observation}
The proof that a giant cycle or  equivalently a non-contractible loop exists (in the case of the torus) follows from the uniqueness of the crossing component. This uniqueness also implies that a cycle cannot ``wind around" the torus multiple times before forming a loop, as this would imply multiple crossing components in all of the boxes.
\end{observation}
\subsection{Duality}\label{sec:duality}

The proofs for $k>1$ will require the following duality between the occupancy and vacancy processes.
\begin{lem}\label{lem:duality}
Recall that $\io:\Hg_k(\cO_r) \to \Hg_k(\T^d)$ and $\iv: \Hg_k(\cV_r) \to \Hg_k(\T^d)$ are the maps (group homomorphisms) induced by the inclusion map. Then,
\[
	\beta_k(\T^d) := \rank(\Hg_k(\T^d)) = \rank(\io) + \rank(\bar i_{d-k}).
\]
\end{lem}
Note that since we are using field coefficients, the homology groups are vector spaces, and we can simply replace $\rank$ with $\dim$.

Recall the definitions of the events $A_k,E_k,\bar A_k,\bar E_k$. The following corollary will be very useful for us.
\begin{cor}\label{cor:duality}
The event $A_k$ occurs if and only if $\bar E_{d-k}$ does not. In other words, $A_k$ and $\bar E_{d-k}$ are complementing events.
\end{cor}
\begin{proof}
The event $A_k$ occurs if and only if $\rank(i_k) = \beta_k(\T^d)$. Using Lemma \ref{lem:duality}, this  holds if and only if $\rank(\bar i_{d-k}) =0$. Finally, by definition $\rank(\bar i_{d-k}) =0$ if and only if $\bar E_{d-k}$ does not hold. This completes the proof.
\end{proof}

The proof of Lemma \ref{lem:duality} requires more familiarity with algebraic topology than the rest of the paper,
 but is not required in order to understand the rest of the paper. We use a form of Alexander duality, which relates the homology of a suitably well-behaved subset of a space with the cohomology of its complement (see \cite{hatcher_algebraic_2002}).

\begin{lem}[\cite{hatcher_algebraic_2002} Thm 3.44]\label{lem:duality_classic}
Let $M$ be a closed orientable $d$-manifold, and let $K\subset M$ be  compact and locally contractible. Then,
\[
 \Hg_{k}(M, M-K) \cong \Hg^{d-k}(K)
\]
\end{lem}
Before continuing we make a few remarks. First, the locally contractible condition follows in our case as the number of balls intersecting any point is finite almost surely. We also note that since we are considering (co)homology over a field, homology and cohomology are dual vector spaces, so their ranks/dimensions are the same.  Finally, we note for the reader that although Alexander duality is most commonly stated for the case $M=\mathbb{S}^d$, it remains true for any compact manifold.

\begin{proof}[Proof of Lemma \ref{lem:duality}]
Take $M=\T^d$ and $K=\cO_r$ (so that $M-K = \cV_r$) in Lemma \ref{lem:duality_classic}, and consider the following diagram,
\[\xymatrix@C+1pc{
    \Hg_{d-k}(\cV_r)\ar[r]^{{\bar i_{d-k}}}
        & \Hg_{d-k}(\T^d)\ar[d]^{\cong}\ar[r]^{j}
        & \Hg_{d-k}(\T^d,\cV_r)\ar[d]^{\cong}\\
        & \Hg^{k}(\T^d)\ar[r]^{i^k}
        & \Hg^{k}(\cO_r)}
\]
The first row in this diagram is a part of the long exact sequence for relative homology. The left vertical map is the isomorphism given by Poincar\'{e} duality, and  the second vertical map is the isomorphism provided by Lemma \ref{lem:duality_classic}. The fact that this diagram commutes arises as part of the proof of Lemma \ref{lem:duality_classic} (see \cite{hatcher_algebraic_2002}).

By the rank-nullity theorem, we have
\eqb\label{eq:rank_null}
\rank(\Hg_{d-k}(\T^d)) = \rank(\ker(j)) + \rank(\im(j)).
\eqe
Since the top row is exact we have that $\ker(j) = \im(\bar i_{d-k})$, implying that
 $\rank(\ker(j)) = \rank(\im(\bar i_{d-k}))$. In addition, since we are assuming field coefficients, and using Poincar\'{e} duality, we have that
\[
\rank(\Hg_{d-k}(\T^d)) = \rank(\Hg^k(\T^d)) = \rank(\Hg_k(\T^d)) = \beta_k(\T^d).
\]
Finally, since the square in the diagram commutes, and both vertical maps are isomorphisms, we have that $\rank(\im(j)) = \rank(\im(i^k))$. Since we assume field coefficients, the rank of the vector space and its dual are the same~\cite{de2011dualities}, and therefore  $\rank(\im(i^k)) = \rank(\im(i_k))$.
Putting all these arguments into \eqref{eq:rank_null} completes the proof.

\end{proof}

\subsection{Giant $(d-1)$-cycles}

The duality in Lemma \ref{lem:duality} and Corollary \ref{cor:duality} imply that if we can prove a phase transition for $\Hg_1(\cV_r)$, it will imply a  phase transition  for $\Hg_{d-1}(\cO_r)$ as they are complementing.  We note that dualities of a similar spirit  have been used for bond percolation in $\R^2$~\cite{grimmett_percolation._nodate}, as well as implicitly in ``blocking surface" arguments~\cite{duminil-copin_sharp_2017,ahlberg2018sharpness}.

Our proof of Lemma \ref{lem:proof_1} shows that the transition for $\Hg_1$ is equivalent to the transition for the giant component.
While uniqueness is known for the giant component in $\cO_r$~\cite{penrose_large_1996}, to the best of our knowledge, to date no proof exists for uniqueness of the giant component in $\cV_r$ (i.e. the equivalent of Proposition~\ref{prop:penrose_sup} for the vacancy). While we expect such statement to be true, there are numerous technical obstacles when dealing with the vacancy process, primarily due to its more complicated geometry (see Figure~\ref{fig:vacancy}).
 Thus, for the time being we make the following weaker statement.

\begin{figure}
\centering
    \begin{subfigure}{0.4\textwidth}
    \centering
    \includegraphics[width=0.85\textwidth]{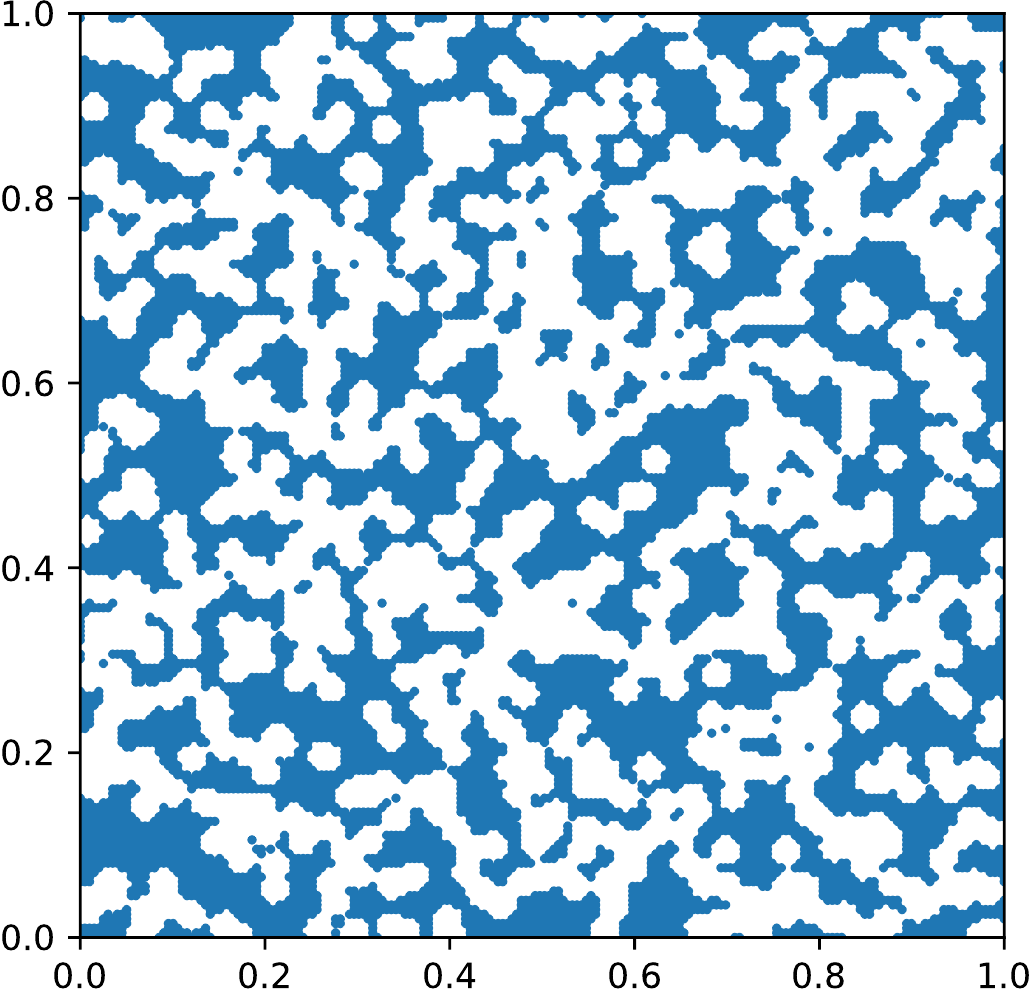}
    \caption{}
    \end{subfigure}
      \hspace{25pt}
      \begin{subfigure}{0.4\textwidth}
      \centering
      \includegraphics[width=\textwidth]{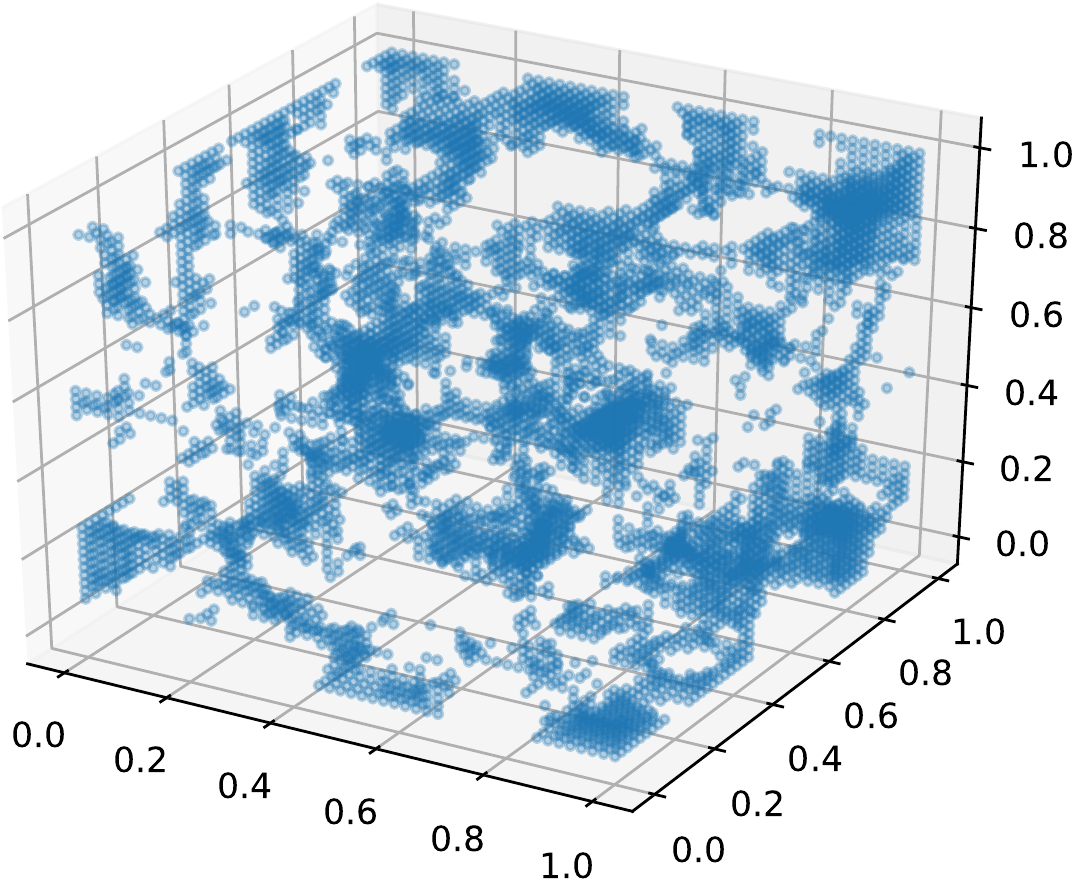}
      \caption{}
      \end{subfigure}
\caption{\label{fig:vacancy} An approximation of the vacancy (shown in blue) at thresholds small enough ($\lambda < \bar \lambda_c$) so that $\bar A_1$ has occurred. (a) In $\T^2$, we have $\bar A_1 = E_1^c$, and therefore we observe no $1$-cycles in the occupancy (white).
 (b) In $\T^3$, we have $\bar A_1 = E_2^c$, and thus the occupancy contains no $2$-cycles. As can be seen, the vacancy has a much more challenging geoemtry as components can be arbitrarily small (whereas in the occupancy, the volume of a component is lower bounded by the volume of a ball, i.e. $\Omega(r^d)$). }
\end{figure}

\begin{lem}\label{lem:proof_3}
The thresholds for the giant ${d-1}$-cycles satisfy
\[
	\lambda_{0,d-1} \le \lambda_{1,d-1} \le \bar \lambda_c < \infty,
\]
where $\bar \lambda_c$ is the percolation threshold for the vacancy process in $\R^d$.
\end{lem}

Before proving the lemma, we require one intermediate technical result.

\begin{lem}\label{lem:box_to_sphere}
Let $c$ be the center point of $Q^d$, and let $B_{R}(c)$ be a ball centered at the origin of radius $R<1/2$, and be $Q_{\eps r}(c)$ be a box of side-length $\eps r$ centered at $c$.

If $\lambda > \bar \lambda_c$ then there exists $C_4>0$ such that
\[
  \prob{Q_{\eps r}(c) \stackrel{\cV_r}{\longleftrightarrow} \partial B_{R}(c) }\le  e^{-C_4 R n^{1/d}}.
\]
\end{lem}

\begin{proof}
From Proposition \ref{prop:exponential_decay} we have when $\lambda > \bar \lambda_c$ we have
\eqb\label{eqn:QtoB}
	\prob{c\stackrel{\cV_r}{\longleftrightarrow} \partial B_{R}(c)} \le e^{-C_2 R n^{1/d}}.
\eqe
Next,
\eqb\label{eqn:ctoQ}
	\prob{Q_{\eps r}(c) \subset \cV_r} \ge \prob{ B_{r(1+\sqrt{d}\eps/2)} \cap \cP_n = \emptyset} = e^{-\lambda\omega_d(1+\sqrt{d}\eps/2)^d} := C
\eqe
Note that if we have that both $Q_{\eps r}(c) \subset \cV_r$ and  $Q_{\eps r}(c) \stackrel{\cV_r}{\longleftrightarrow} \partial B_{R}(c)$, then  necessarily $c\stackrel{\cV_r}{\longleftrightarrow} \partial B_{R}(c)$. Thus,
\[
	\prob{Q_{\eps r}(c) \subset \cV_r \text{ and } Q_{\eps r}(c) \stackrel{\cV_r}{\longleftrightarrow} \partial B_{R}(c)} \le \prob{c\stackrel{\cV_r}{\longleftrightarrow} \partial B_{R}(c)} \le e^{-C_2R n^{1/d}}.
\]
Since both events on the LHS are decreasing, we can use the FKG inequality (see, e.g.~\cite{meester_continuum_1996})  together with \eqref{eqn:ctoQ} and have
\[
	C\cdot \prob{Q_{\eps r}(c) \stackrel{\cV_r}{\longleftrightarrow} \partial B_{R}(c)}\le e^{-C_2 R n^{1/d}}.
\]
Thus, we can find $C_4>0$ such that $\prob{Q_{\eps r}(c) \stackrel{\cV_r}{\longleftrightarrow} \partial B_{R}(c)}\le e^{-C_4 R n^{1/d}}$, completing the proof.

\end{proof}

\begin{proof}[Proof of Lemma \ref{lem:proof_3}]
Suppose that $\lambda > \bar \lambda_c$. The proof is mostly similar to the proof of Lemma \ref{lem:proof_1}.
The main difference here is that there is no discretization $Q^d \cap( \eps r\cdot\Z^d)$ that guarantees that  a component in $\cV_r$ will intersect any of the grid points. Instead, for every $x$ in the grid we take $Q_{\eps  r}(x)$ to be the box of side-length $\eps r$ centered at $x$. Since the union of these boxes covers $Q^d$, we have that every component in $\cV_r$ must intersect at least one of these boxes. As in the proof of Lemma \ref{lem:proof_1} we argue that if $\gamma \subset \cV_r$ is a realization of a non-trivial $1$-cycle in $\Hg_1(\cV_r)$ that is mapped to a non-trivial cycle in $\Hg_1(\T^d)$, then $\gamma$ cannot be contained in a ball of radius $R<1/2$. For any point $x\in \T^d$, using the translation-invariance of the torus, and  Lemma \ref{lem:box_to_sphere}, we have that
\[
\prob{Q_{\eps r}(x) \stackrel{\cV_r}{\longrightarrow} \partial B_R(x)} \le e^{-C_4 R n^{1/d}}.
\]
Since we have $M= (\eps r)^{-d} = O(n)$ boxes, using a union bound, we have
\[
	\prob{\bar E_1} \le Me^{-C_4 R n^{1/d}} = O\param{ne^{-C_4 R n^{1/d}}}.
\]
Using Corollary \ref{cor:duality}, we have that $A_{d-1} = {\bar E_1}^c$. Thus, we have
\[
	\prob{A_{d-1}} \ge 1-Me^{-C_4 R n^{1/d}},
\]
implying that $\lambda > \lambda_{1,d-1}$. Therefore, we conclude that $\lambda_{1,d-1}\le \bar \lambda_c$, completing the proof.
	\end{proof}

\subsection{Giant $k$-cycles, $1<k< d-1$}
In this section we will prove that the appearance of all giant $k$-cycles ($1<k<d-1$) occurs between $\lambda_c$ and $\bar \lambda_c$, and in an increasing order, as staged in Theorem \ref{thm:main}.
The following lemma is the main result of this section.

\begin{lem}\label{lem:proof_2}
For every $1\le k \le d-2$ and $\theta=0,1$,  we have $\lambda_{\theta,k} \le \lambda_{\theta,k+1}$.
\end{lem}

To prove Lemma \ref{lem:proof_2}, we will use the following statement, which is a consequence of the duality in Lemma \ref{lem:duality}.
\begin{lem}\label{lem:inclusions}
The events $A_k,E_k$ satisfy
\[
A_1 \supset A_2 \supset \cdots\supset A_{d-1}\quad\text{and}\quad E_1 \supset E_2 \supset \cdots\supset E_{d-1},
\]
and the same holds for $\bar A_k,\bar E_k$.
\end{lem}

\begin{proof}
For $i=1,\ldots,d$ define $\T^d_i:= (\R^{i-1}\times\set{0}\times \R^{d-i})/\Z^d$. In other words, $\T^d_i$ are $(d-1)$-dimensional flat tori embedded in $\T^d$. Let $\cO_r^{(i)} := \cO_r\cap \T_i^d$, and $\cV_r^{(i)} := \cV_r\cap\T_i^d$ be the induced (or projected) processes. Similarly to the events $E_k,A_k$ we can define $E^{(i)}_k,A_k^{(i)},\bar E^{(i)}_k,\bar A_k^{(i)}$ ($1\le k \le d-2$) with respect to the processes $\cO_r^{(i)},\cV_r^{(i)}$, and the $(d-1)$-torus $\T_i^d$.

Fix $1\le k\le d-2$, and suppose that $A_k^{(i)}$ occurs, then
\eqb\label{eqn:A_k_i}
	\im(\Hg_k(\cO_r^{(i)})\to \Hg_k(\T^d_i)) = \Hg_k(\T^d_i).
\eqe
We now require two topological facts:
\begin{enumerate}
  \item The inclusion  $\T_i^d\hookrightarrow \T^d$ induces an injective map in homology, i.e. the map
  $\Hg_k(\T_i^d)\to \Hg_k(\T^d)$ is injective;
  \item For $k<d$, the $k$-dimensional classes in $\T^d$ are spanned by the $k$-dimensional classes in the $d$ subtorii, $\T^d_i$, i.e. $\sum_{i=1}^d \Hg_k(\T^d_i) = \Hg_k(\T^d)$, where the summation represents the sum of the vector spaces as subspaces of  $\Hg_k(\T^d)$.
\end{enumerate}
These two results are well-known. However, for completeness we include proofs in Appendix~\ref{sec:appendix1}.
The fact that $\Hg_k(\T_i^d)\to \Hg_k(\T^d)$ is injective, implies that
\[
	\im(\Hg_k(\cO_r^{(i)})\to \Hg_k(\T^d_i)) \cong \im(\Hg_k(\cO_r^{(i)})\to \Hg_k(\T^d)).
\]
 If $A_k^{(i)}$ occur for all $i=1,\ldots,d$ we have
\[
	\Hg_k(\T^d) = \sum_{i=1}^d \Hg_k(\T^d_i) \cong \sum_{i=1}^d \im(\Hg_k(\cO_r^{(i)})\to \Hg_k(\T^d))\subset \im(\Hg_k(\cO_r)\to \Hg_k(\T^d)) \subset \Hg_k(\T^d),
\]
implying that the last relation is an equality,
so that $A_k$ holds as well.
In other words, we showed that
\eqb \label{eqn:inc_A}
	A_{k}^{(1)} \cap\cdots\cap A_k^{(d)}\subset A_k,
\eqe
and similarly we can show that
\eqb\label{eqn:inc_E}
	E_{k}^{(1)} \cup\cdots\cup E_k^{(d)}\subset E_k.
\eqe
The same inclusions will apply for $\bar A_k, \bar E_k$.

Next, from Corollary \ref{cor:duality} we have that $A_k = (\bar E_{d-k})^c$, and since $\T^d_i$ is a $(d-1)$-torus, we also have that $ A_k^{(i)} = (\bar E_{d-1-k}^{(i)})^c$. Putting all these connections together we have that
\[
\splitb
	A_k &\supset  \param{A_k^{(1)}\cap\cdots\cap A_k^{(d)}}= \param{(\bar E_{d-1-k}^{(1)})^c \cap\cdots\cap (\bar E_{d-1-k}^{(d)})^c }
	\supset (\bar E_{d-1-k})^c = A_{k+1},
\splite
\]
and
\[
\splitb
	E_k &\supset  \param{E_k^{(1)}\cup\cdots\cup E_k^{(d)}}= \param{(\bar A_{d-1-k}^{(1)})^c \cup\cdots\cup (\bar A_{d-1-k}^{(d)})^c }
	\supset (\bar A_{d-1-k})^c = E_{k+1}.
\splite
\]
Similarly, we can prove that $\bar A_k \supset \bar A_{k+1}$ and $\bar E_k \supset \bar E_{k+1}$, concluding the proof.
\end{proof}

\begin{proof}[Proof of Lemma \ref{lem:proof_2}]
For any $\lambda < \lambda_{0,k}$ we have $\limsup_{\ninf} n^{-1/d}\log \prob{E_k} < 0$. From Lemma \ref{lem:inclusions} we have that $E_{k+1}\subset E_k$, and therefore we also have $\limsup_{\ninf} n^{-1/d}\log \prob{E_{k+1}} < 0$, implying that $\lambda < \lambda_{0,k+1}$. Thefeore, we conclude that $\lambda_{0,k}\le \lambda_{0,k+1}$.

Similarly, for all $\lambda > \lambda_{1,k+1}$ we have $\limsup_{\ninf} n^{-1/d}\log (1- \prob{A_{k+1}})< 0$, and from Lemma \ref{lem:inclusions} we have $\limsup_{\ninf} n^{-1/d}\log (1- \prob{A_{k}})< 0$, implying that $\lambda > \lambda_{1,k}$. Therefore, $\lambda_{1,k} \le \lambda_{1,k+1}$. This completes the proof.
\end{proof}
\section{Discussion}\label{sec:discuss}

In this paper we have defined a notion of giant $k$-dimensional cycles that emerge in a continuum percolation model on the torus.  We have shown the existence of thresholds for the appearance of these cycles and that these thresholds are in the thermodynamic limit. In this section we provide some insights and directions for future work.

\begin{itemize}
\item The main open problem  remains proving Conjecture \ref{conj}, i.e.~that all transitions are sharp, and that the ordering between the thresholds is strict. As we stated earlier, to prove sharpness for $k=d-1$ all that is required is a uniqueness statement for the giant vacancy component. For the intermediate dimensions, it is less clear how to prove sharpness.

\item A desirable extension would be to state and prove analogous results for general manifolds as well as the appearance of the fundamental group. In the case of the torus, the sharp threshold for the fundamental group follows directly from the homological statements in this paper. The main challenges in manifolds will be that we must deal with (a) curvature, and (b) the existence and representability of giant cycles. In principle, we do not expect curvature to change these statements, however it adds significant technical complications (see \cite{bobrowski_random_2019}). The lack of a product structure in general manifolds makes relating giant cycles of different dimensions more difficult as well. However, we note that these do not apply to the fundamental group.

\item In this paper we used balls with a fixed radius $r$. The most common model studied in continuum percolation is where the grains are balls with \emph{random} radii. Most of the statements in this paper can be translated to the random-radii case (assuming bounded moments). However, the equality $\lambda_{0,1} = \lambda_{1,1} = \lambda_c$ requires a quantitative uniqueness of the crossing component, i.e. bounds on the second largest component, which to the best of our knowledge, has not been proved for the general case.

\item In a recent paper \cite{bobrowski2020homological}, we experimentally studied the homological percolation thresholds in various models including continuum percolation, site percolation, and Gaussian random fields. We compared these thresholds to the zeros of the expected Euler characteristic curve (as a function of $\lambda$), which has an explicit expression. The simulation results in \cite{bobrowski2020homological} show that the percolation thresholds always appear very near the zeros of the expected EC. This is a somewhat surprising result, as the EC is a quantitive descriptor (counting cycles), while the percolation thresholds describe a qualitative phenomenon (the emergence of giant cycles). It remains an open question as to the nature of this observed correlation, and whether  the zeros of the EC curve (which can be evaluated analytically) can potentially be used to approximate or at least bound the percolation thresholds.

\item The definitions we used here for giant cycles, can be applied in the context on various other percolation models such as bond and site percolation. In principle, the general behavior should follow similarly to the one we observed here, while the proof might require a slightly different approach.  In particular, the duality statement we have here, does not apply directly to other models.

\item In the applied topology aspect of this work, recall that we wish to distinguish between the signal and noise cycles in persistent homology. For every signal (giant) cycle $\gamma_{\text{signal}}$, we have $\death(\gamma_{\text{signal}}) = \text{const}$, while the results in this paper imply that $\birth(\gamma) = \text{const}\cdot n^{-1/d}$ (since the giant cycles are formed when $nr^d = \lambda$). Therefore, the persistence value for the giant cycles, satisfies
\[
\pi(\gamma_{\text{signal}}) := \frac{\death(\gamma_{\text{signal}})}{\birth(\gamma_{\text{signal}})} =  \Theta(n^{1/d}).
\]
In \cite{bobrowski_maximally_2017} it was shown that the persistence of all the noise cycles satisfies
\[
\pi(\gamma_{\text{noise}})  = O\param{\param{\frac{\log n}{\logg n}}^{1/k}}.
\]
In other words, the results in this paper indicate that asymptotically the persistence of the signal and the noise cycles differ by orders of magnitudes. From the applied topology perspective, this is an optimistic statement, since it means that given a large sample, we could use persistence to distinguish between signal and noise.
\end{itemize}

\bibliographystyle{plain}
\bibliography{zotero}

\begin{thebibliography}{10}

\bibitem{adler_modeling_2017}
Robert~J. Adler, Sarit Agami, and Pratyush Pranav.
\newblock Modeling and replicating statistical topology and evidence for {CMB}
  nonhomogeneity.
\newblock {\em Proceedings of the National Academy of Sciences},
  114(45):11878--11883, November 2017.

\bibitem{ahlberg2018sharpness}
Daniel Ahlberg, Vincent Tassion, and Augusto Teixeira.
\newblock Sharpness of the phase transition for continuum percolation in
  $\mathbb{R}^2$.
\newblock {\em Probability Theory and Related Fields}, 172(1-2):525--581, 2018.

\bibitem{aizenman1987sharpness}
Michael Aizenman and David~J Barsky.
\newblock Sharpness of the phase transition in percolation models.
\newblock {\em Communications in Mathematical Physics}, 108(3):489--526, 1987.

\bibitem{aizenman_sharp_1983}
Michael Aizenman, J.~T. Chayes, Lincoln Chayes, J.~Fr{\"o}hlich, and L.~Russo.
\newblock On a sharp transition from area law to perimeter law in a system of
  random surfaces.
\newblock {\em Communications in Mathematical Physics}, 92(1):19--69, 1983.

\bibitem{auffinger_topologies_2018}
Antonio Auffinger, Antonio Lerario, and Erik Lundberg.
\newblock Topologies of random geometric complexes on {Riemannian} manifolds in
  the thermodynamic limit.
\newblock {\em arXiv preprint arXiv:1812.09224}, 2018.

\bibitem{bobrowski2019homological}
Omer Bobrowski.
\newblock Homological connectivity in random {{\v C}}ech complexes, 2019.

\bibitem{bobrowski_maximally_2017}
Omer Bobrowski, Matthew Kahle, and Primoz Skraba.
\newblock Maximally persistent cycles in random geometric complexes.
\newblock {\em The Annals of Applied Probability}, 27(4):2032--2060, 2017.

\bibitem{bobrowski_random_2019}
Omer Bobrowski and Goncalo Oliveira.
\newblock Random {{\v C}}ech {Complexes} on {Riemannian} {Manifolds}.
\newblock {\em Random Structures \& Algorithms}, 54(3):373--412, 2019.

\bibitem{bobrowski2020homological}
Omer Bobrowski and Primoz Skraba.
\newblock Homological percolation and the {E}uler characteristic.
\newblock {\em Physical Review E}, 101(3):032304, 2020.

\bibitem{borsuk_imbedding_1948}
Karol Borsuk.
\newblock On the imbedding of systems of compacta in simplicial complexes.
\newblock {\em Fundamenta Mathematicae}, 35(1):217--234, 1948.

\bibitem{broadbent1957percolation}
Simon~R Broadbent and John~M Hammersley.
\newblock Percolation processes: I. crystals and mazes.
\newblock In {\em Mathematical Proceedings of the Cambridge Philosophical
  Society}, volume~53, pages 629--641. Cambridge University Press, 1957.

\bibitem{carlsson_topology_2009}
Gunnar Carlsson.
\newblock Topology and data.
\newblock {\em Bulletin of the American Mathematical Society}, 46(2):255--308,
  2009.

\bibitem{de2011dualities}
Vin De~Silva, Dmitriy Morozov, and Mikael Vejdemo-Johansson.
\newblock Dualities in persistent (co) homology.
\newblock {\em Inverse Problems}, 27(12):124003, 2011.

\bibitem{duminil-copin_sixty_2017}
Hugo Duminil-Copin.
\newblock Sixty years of percolation.
\newblock {\em arXiv preprint arXiv:1712.04651}, 2017.

\bibitem{duminil-copin_sharp_2017}
Hugo Duminil-Copin, Aran Raoufi, and Vincent Tassion.
\newblock Sharp phase transition for the random-cluster and {Potts} models via
  decision trees.
\newblock {\em arXiv preprint arXiv:1705.03104}, 2017.

\bibitem{duminil-copin_subcritical_2018}
Hugo Duminil-Copin, Aran Raoufi, and Vincent Tassion.
\newblock Subcritical phase of $d$-dimensional {Poisson}-{Boolean} percolation
  and its vacant set.
\newblock {\em arXiv:1805.00695 [math-ph]}, May 2018.
\newblock arXiv: 1805.00695.

\bibitem{edelsbrunner_persistent_2008}
Herbert Edelsbrunner and John Harer.
\newblock Persistent homology-a survey.
\newblock {\em Contemporary mathematics}, 453:257--282, 2008.

\bibitem{edelsbrunner_computational_2010}
Herbert Edelsbrunner and John~L. Harer.
\newblock {\em Computational topology: an introduction}.
\newblock AMS Bookstore, 2010.

\bibitem{erdos_evolution_1960}
Paul Erd{\"o}s and Alfr{\'e}d R{\'e}nyi.
\newblock On the evolution of random graphs.
\newblock {\em Publ. Math. Inst. Hung. Acad. Sci}, 5(17-61):43, 1960.

\bibitem{ghrist2014elementary}
Robert~W Ghrist.
\newblock {\em Elementary applied topology}.
\newblock Createspace Seattle, 2014.

\bibitem{gilbert_random_1961}
Edward~N. Gilbert.
\newblock Random plane networks.
\newblock {\em Journal of the Society for Industrial \& Applied Mathematics},
  9(4):533--543, 1961.

\bibitem{grimmett_percolation._nodate}
G.~Grimmett.
\newblock {\em Percolation. 1999}.
\newblock Springer, Berlin. Math. Review 2001a.

\bibitem{grimmett_plaquettes_2010}
Geoffrey~R. Grimmett and Alexander~E. Holroyd.
\newblock Plaquettes, spheres, and entanglement.
\newblock {\em Electron. J. Probab}, 15:1415--1428, 2010.

\bibitem{hatcher_algebraic_2002}
Allen Hatcher.
\newblock {\em Algebraic topology}.
\newblock Cambridge University Press, Cambridge, 2002.

\bibitem{hiraoka_percolation_2018}
Yasuaki Hiraoka and Tatsuya Mikami.
\newblock Percolation on {Homology} {Generators} in {Codimension} {One}.
\newblock {\em arXiv preprint arXiv:1809.07490}, 2018.

\bibitem{hiraoka_limit_2018}
Yasuaki Hiraoka, Tomoyuki Shirai, and Khanh~Duy Trinh.
\newblock Limit theorems for persistence diagrams.
\newblock {\em The Annals of Applied Probability}, 28(5):2740--2780, 2018.

\bibitem{kahle_random_2011}
Matthew Kahle.
\newblock Random geometric complexes.
\newblock {\em Discrete \& Computational Geometry}, 45(3):553--573, 2011.

\bibitem{linial_phase_2016}
Nathan Linial and Yuval Peled.
\newblock On the phase transition in random simplicial complexes.
\newblock {\em Annals of Mathematics. Second Series}, 184(3):745--773, 2016.

\bibitem{meester_continuum_1996}
Ronald Meester and Rahul Roy.
\newblock {\em Continuum percolation}, volume 119.
\newblock Cambridge University Press, 1996.

\bibitem{menshikov1986coincidence}
Mikhail~V Menshikov.
\newblock Coincidence of critical points in percolation problems.
\newblock In {\em Soviet Mathematics Doklady}, volume~33, pages 856--859, 1986.

\bibitem{owada2020convergence}
Takashi Owada, Omer Bobrowski, et~al.
\newblock Convergence of persistence diagrams for topological crackle.
\newblock {\em Bernoulli}, 26(3):2275--2310, 2020.

\bibitem{penrose_random_2003}
Mathew Penrose.
\newblock {\em Random geometric graphs}.
\newblock Oxford University Press Oxford, 2003.

\bibitem{penrose_large_1996}
Mathew~D. Penrose and Agoston Pisztora.
\newblock Large deviations for discrete and continuous percolation.
\newblock {\em Advances in applied probability}, pages 29--52, 1996.

\bibitem{roy1990russo}
Rahul Roy et~al.
\newblock The russo-seymour-welsh theorem and the equality of critical
  densities and the ``dual" critical densities for continuum percolation on
  $\mathbb{\R}^2$.
\newblock {\em The Annals of Probability}, 18(4):1563--1575, 1990.

\bibitem{sarkar_co-existence_1997}
Anish Sarkar.
\newblock Co-existence of the occupied and vacant phase in {Boolean} models in
  three or more dimensions.
\newblock {\em Advances in Applied Probability}, pages 878--889, 1997.

\bibitem{thomas2019functional}
Andrew~M Thomas and Takashi Owada.
\newblock Functional limit theorems for the euler characteristic process in the
  critical regime.
\newblock {\em arXiv preprint arXiv:1910.00751}, 2019.

\bibitem{wasserman_topological_2018}
Larry Wasserman.
\newblock Topological data analysis.
\newblock {\em Annual Review of Statistics and Its Application}, 5, 2018.

\bibitem{yogeshwaran_random_2016}
D.~Yogeshwaran, Eliran Subag, and Robert~J. Adler.
\newblock Random geometric complexes in the thermodynamic regime.
\newblock {\em Probability Theory and Related Fields}, pages 1--36, 2016.

\bibitem{zomorodian_topology_2005}
Afra~J. Zomorodian.
\newblock {\em Topology for computing}, volume~16.
\newblock Cambridge university press, 2005.

\end{thebibliography}
\appendix
\section{A topological supplement for the proof of Lemma~\ref{lem:inclusions}}\label{sec:appendix1}
Recall that $\T^d$ is the flat $d$-torus and $\T_i^d$ is the $(d-1)$-torus defined by the $i$-th flat. First,  we show that the inclusion map	$\T_i^d\hookrightarrow \T^d$, induces an injective map on homology
$$\Hg_k(\T_i^d)\hookrightarrow \Hg_k(\T^d),$$
Express the $d$-torus as the $d$-fold cartesian product of circles
	$\T^d = \mathbb{S}^1\times \ldots\times \mathbb{S}^1$.
Hence, we can rewrite $\T^d$ as the cartesian product
$\T^d = \T_i^{d}\times \mathbb{S}^1$.
Taking homology, we have the following isomorphism via the K\"{u}nneth formula,
$$ \bigoplus\limits_{k+\ell=m} \Hg_k(\T_i^{d})\otimes \Hg_\ell(\mathbb{S}^1) \cong \Hg_k(\T^{d})$$
Restricting to $\ell=0$ yields the required injective map. Note that this follows from the fact that there is no torsion since we are working over field coefficients.

Next we show that  for $k<d$
$$ \sum\limits_{i=1}^d \Hg_k(\T_i^d) = \Hg_k(\T^d).$$
where $\Hg_k(\T_i^d)$ are taken as vector subspaces of the vector space $\Hg_k(\T^d)$. This is well defined since the maps are injective by the argument above.

Hence, we can take the sum of the individual vector spaces as subspaces, denoted by
$ \sum\limits_{i=1}^d \Hg_k(\T_i^d)$. For each $i$, $\Hg_k(\T_i^d)\subseteq \Hg_k(\T^d)$, so it follows that
$$\sum\limits_{i=1}^d \Hg_k(\T_i^d)\subseteq \Hg_k(\T^d).$$

In the other direction, on can again use the representation of $\T^d$ as the Cartesian product of circles. Applying the K\"{u}nneth formula $d$ times, we obtain
 $$\Hg_k(\T^d) = \sum\limits_{\sum_{i} k(j) = k } \Hg_{k(1)}(\mathbb{S}^1) \otimes  \cdots \otimes \Hg_{k(d)}(\mathbb{S}^1) $$
As the homology of $\mathbb{S}^1$ is only non-zero for $k=0,1$, a $k$-cycle in $\T^d$ can be represented by taking the $\Hg_1(\mathbb{S}^1)$ in some $k$ coordinates and $\Hg_0(\mathbb{S}^1)$
in the others. This is an element of any $\Hg_k(\T^d_i)$ where $k(i)=0$ and since $k<d$, there must be at least one. The result follows.
\end{document}